\RequirePackage{fix-cm}

\documentclass[smallextended,a4paper]{my-svjour3} 
\smartqed

\usepackage{amssymb,amsmath,amsfonts,latexsym}
\usepackage{hyperref}

\def\proofend{\hfill$\Box$} 

\usepackage{mathtools}
\usepackage{multirow,array}
\usepackage{graphicx}
\usepackage{subfigure}
\usepackage{epstopdf}
\usepackage{float} 

\allowdisplaybreaks

\journalname{Fract. Calc. Appl. Anal.} 

\begin{document}

\title{Existence and uniqueness of mild solutions\\ 
for a class of psi-Caputo time-fractional systems\\ 
of order from one to two}

\titlerunning{Existence and uniqueness of mild solutions for a class of psi-Caputo \dots}

\author{Hamza Ben Brahim$^1$ 
\and\\
Fatima-Zahrae El Alaoui$^2$ 
\and\\
Asmae Tajani$^3$ 
\and\\
Delfim F. M. Torres$^4$} 

\authorrunning{H. Ben Brahim \and  F.Z. El Alaoui \and A. Tajani \and D.F.M. Torres} 

\institute{Hamza Ben Brahim$^{1}$
\at
TSI Team, Department of Mathematics, Faculty of Sciences,\\ 
Moulay Ismail University, 11201 Meknes, Morocco \\
\email{brahim.hamzaben@edu.umi.ac.ma} 
\and
Fatima-Zahrae El Alaoui$^{2}$
\at
TSI Team, Department of Mathematics, Faculty of Sciences,\\ 
Moulay Ismail University, 11201 Meknes, Morocco \\
\email{f.elalaoui@umi.ac.ma}
\and
Asmae Tajani$^{3}$
\at
Ecole Normale Superieur, University Hassan II, Casablanca, Morocco\\ 
\& Center for Research and Development in Mathematics and Applications (CIDMA),\\
Department of Mathematics, University of Aveiro, 3810-193 Aveiro, Portugal\\ 
\email{tajaniasmae@ua.pt}
\and
Delfim F. M. Torres$^{4,*}$ 
\at
Center for Research and Development in Mathematics and Applications (CIDMA),\\
Department of Mathematics, University of Aveiro, 3810-193 Aveiro, Portugal\\
\email{delfim@ua.pt} (corresponding author)}


\date{Received: 2 February 2024 / Revised: 9 and 17 July 2025 / Accepted: 18 July 2025}

\maketitle


\begin{abstract}
We prove the existence and uniqueness of mild solutions for a specific class 
of time-fractional $\psi$-Caputo evolution systems with a derivative order 
ranging from 1 to 2 in Banach spaces. By using the properties of cosine 
and sine family operators, along with the generalized Laplace transform, 
we derive a more concise expression for the mild solution. This expression 
is formulated as an integral, incorporating Mainardi's Wright-type function. 
Furthermore, we provide various valuable properties associated with 
the operators present in the mild solution. Additionally, employing the 
fixed-point technique and Gr\"{o}nwall's inequality, we establish the 
existence and uniqueness of the mild solution. To illustrate our results, 
we conclude with an example of a time-fractional equation, presenting 
the expression for its corresponding mild solution.




\end{abstract} 


\section{Introduction} 
\label{sec:1}

The exploration of fractional calculus has attracted significant attention 
in recent decades, driven by its wide-ranging applications across various fields. 
This interest is primarily motivated by its capability to naturally extend classical 
calculus to non-integer orders, a crucial aspect for modeling complex phenomena. 
In the literature, the definition of fractional differential operators varies, 
encompassing operators such as Hilfer, Riesz, Hadamard, Caputo, Riemann--Liouville, 
among others \cite{kilbas.2006,oldham.1994,podlubny.1998,kkpodlubn,MyID:533}. 
However, the two most 
commonly employed definitions are Riemann--Liouville and Caputo. 
The Riemann--Liouville approach involves limit values of fractional derivatives 
in the initial condition, which can pose challenges in interpretation. 
Conversely, the Caputo fractional derivative is preferred in physical models 
due to its ability to offer a clear interpretation of initial conditions 
and its suitability for measurement \cite{K.Diethelmv}.

The increasing variety of definitions for fractional operators, particularly 
with the emergence of new forms, necessitates an exploration of potential 
connections between them. Traditionally, each fractional operator has been 
studied in isolation, often accompanied by new definitions, properties, 
and sometimes modifications of existing proofs. Consequently, establishing 
a unifying framework to link these seemingly distinct operators is of great 
interest, as it could reduce redundancy and deepen our understanding 
of fractional calculus. One important class of fractional operators includes 
fractional derivatives and integrals defined with respect to another function. 
In this context, such a fractional integral was first introduced by 
Holmgren (1866), building upon Liouville’s earlier work (1835). 
The properties and other aspects of this integral were later examined by Sewell (1937), 
Shelkovnikov (1951), and Chrysovergis (1971). Meanwhile, Erdélyi introduced 
his fractional operators in the 1960s \cite{Erdélyi1,Erdélyi2}, 
and Thomas J.~Osler further developed these concepts through a series of papers 
in the 1970s \cite{T.J.Osler1,T.J.Osler2,T.J.Osler3}. Recently, there has been 
renewed interest in these kind of operators, particularly in their extended form as 
\(\psi\)-Caputo operators. Notably, Almeida \cite{Almeidaa} proposed a generalized 
definition of the Caputo fractional derivative, where the derivative is taken 
with respect to another function \(\psi\). This approach offers greater precision 
in mathematical modeling, as selecting an appropriate function \(\psi\) can 
significantly enhance the accuracy of the model \cite{MyID:526}. In physical 
applications, fractional derivatives with respect to another function have 
been employed to generalize Scott--Blair models with time-varying viscosity 
\cite{Herrmann2}. Similarly, a new formulation of the fractional Dodson diffusion 
equation using these operators was explored in \cite{Herrmann1} to better 
capture memory effects in complex diffusion phenomena. In \cite{Almeid22}, 
it was demonstrated that incorporating fractional derivatives with respect 
to another function provides a more suitable framework for modeling the GDP 
growth rate in the USA.  Furthermore, O.~P.~Agrawal introduced a fractional 
derivative involving two functions-weight and scale functions \cite{Agrawal1}, 
which has inspired additional studies. These include applications such as the 
fractional diffusion equation \cite{Yufeng52}, the fractional advection-diffusion 
model for solute transport in aquifers \cite{Yufeng50}, and the generalized 
Burgers equation, which describes natural phenomena such as traffic 
flow and gas dynamics \cite{Yufeng51}.

For further exploration of fractional calculus and its applications, interested readers 
can refer to the extensive literature, including works such as 
\cite{Bal.Hand.201,chen.2010,koch.1998,Tarasov.Hand.2019}.

Fractional systems have attracted substantial attention in the literature, 
owing to their significance in modeling various physical phenomena 
like diffusion, wave propagation, and viscoelasticity \cite{M.Bolognaaaa,M.Bologna,M.Bolognaa}. 
Currently, numerous researchers are investigating various aspects of fractional differential equations, 
covering topics such as the existence and uniqueness of solutions, 
methods for obtaining exact, explicit, and numerical solutions.
To establish the existence and uniqueness of solutions, mathematicians 
commonly employ the fixed-point theorem, upper-lower solutions, iterative methods, 
and numerical approaches. A pivotal focus lies in the study of fractional evolution equations, 
serving as an abstract framework for a broad spectrum of problems, 
particularly those involving systems evolving over time.
Considerable effort has been dedicated to exploring solutions for these equations. 
Researchers employ diverse tools, such as fixed-point theory, semigroups, 
cosine and sine families operators, and the Laplace transform. For instance, Zhou and Jiao \cite{14} 
investigated the mild solutions of a fractional evolution equation with a Caputo fractional 
derivative, utilizing the Laplace transform and semigroup theory. In a related study, 
Shu and Wang demonstrated the existence and uniqueness of a class of nonlocal fractional systems, 
using the theory of cosine and sine families operators and the Krasnoselskii fixed-point theorem \cite{shu.2012}.
Additionally, there is significant interest in studying fundamental solutions for homogeneous 
fractional evolution equations. Yang establishes suitable conditions 
to ensure the uniqueness and existence of smooth solutions for a specific class of stochastic 
evolution equations with variable delay, characterized by the $\psi$-Caputo fractional derivative, 
using the fixed-point technique, stochastic analysis, and semigroups \cite{Yangyangy}. Furthermore, 
Suechoei and Ngiamsunthorn explored the local and global existence 
and uniqueness of mild solutions to a fractional evolution equation of a specific form, 
generalizing the Laplace transform and probability density functions \cite{Yangyangyyyyyyy}. 
For further literature on mild solutions and their applications, we refer to 
\cite{belmekki.2010,eid.2004,MR4682896,pazy.2020,williams.2021,zhou.2016} 
and the references therein.

Inspired by the works \cite{Yangyangyyyyyyy,Yangyangy}, 
we consider here the following fractional evolution system:
\begin{equation}
\label{k}
\begin{cases}
\mathcal{^CD}^{\alpha ,\psi}_{0^{+}}\Theta(t)
=A \Theta(t)+f(t,\Theta(t)),  \quad t\in \ {[0,T]}, \\	 
\Theta(0)=\Theta_{0}  \ ; \ \dfrac{d}{d t} \Theta(0)=\Theta_{1},
\end{cases}
\end{equation}
where $\psi$ is a strictly increasing function with continuous derivative 
$\psi^\prime$ defined on the interval $[0,T]$, $\mathcal{^CD}^{\alpha ,\psi}_{0^{+}}$ 
is the $\psi$-Caputo fractional derivative of order $1< \alpha\leq 2$, 
$A$ is an infinitesimal generator of a strongly continuous cosine family 
of uniformly bounded linear operators  
${(C(t))}_{t\in {\mathbb{R}}}$ on a Banach space $X$,  
and the nonlinear function $f$ is in  $L^2(]0, T]\times X; X)$. 

The dynamics of system \eqref{k} differs significantly between the subdiffusive 
regime \(\alpha \in (0,1]\) and the superdiffusive regime \(\alpha \in (1,2]\). 
In the subdiffusive case, strong memory effects lead to slow information propagation 
and non-exponential relaxation, often associated with anomalous diffusion in complex media. 
In contrast, the superdiffusive regime exhibits faster diffusion, where inertia-like effects 
can cause transient oscillations and wave-like behavior. As \(\alpha\) increases, 
the system transitions from gradual relaxation to more dynamic responses, influencing 
stability and physical interpretation, particularly in applications 
like viscoelastic media and wave propagation.

The article presents a novel and explicit expression for the mild solution 
of system \eqref{k}, overcoming limitations associated with existing approaches 
that rely on Mainardi's Wright-type function, which is traditionally restricted 
to \(\alpha \in ]0,1[\). By conducting a rigorous analysis, the study constructs 
a new solution operator that allows for the formulation of mild solutions beyond 
this constraint, thereby extending the applicability of fractional differential 
equations. Additionally, the article establishes a compactness result for the 
solution operators under the assumption that the cosine family associated with 
\( A \) is compact, providing valuable insights into the structural properties 
and long-term behavior of the system.

This work is organized as follows. In Section~\ref{sec:02}, 
we recall the fundamental mathematical tools that we will 
use in our main results. In Section~\ref{sec:03}, we construct 
a new and specific form of the mild solution of the system \eqref{k} 
using a probability density. In addition, we give and prove  
several useful properties related to the operators found in the mild solution. 
Then, in Section~\ref{sec:04}, we prove the existence and uniqueness 
of the mild solution of the system \eqref{k} using two fixed-point techniques 
under different assumptions in the data of the problem. 
In Section~\ref{sec:05}, we present a formula for a mild solution using 
the Mittag-Leffler function where $A$ is a second-order operator. Finally, 
we give an example to illustrate the applicability of the theoretical result 
in Section~\ref{sec:06}. We end with a conclusion in Section~\ref{sec:07}. 


\section{Preliminaries}
\label{sec:02}

In this section, we introduce some preliminaries of the cosine and sine families, 
related to an operator, as well as fractional integration and differentiation  
that it will be used throughout the work. 

For $T>0$, let $J=[0,T]$ be a time interval. We take $X$ to be a Banach space endowed 
with the norm $\|\cdot\|_X$. We denote by $L(X)$ the space of all linear bounded 
operators from $X$ to itself with the norm $\|\cdot\|_{L(X)}$, and by $C(J;X)$ 
we denote the space of continuous functions from $J$ into $X$ endowed with the norm 
$\| z \| = \displaystyle\sup_{t\in J} \|z(t)\|_X$.  

For a linear unbounded operator $A : \mathcal{D}(A)\subset X \rightarrow X $ 
with $\mathcal{D}(A)$ its domain, we denote by $\rho(A)$ and $R(\lambda ; A) 
= {(\lambda I -A)}^{-1}$, where $\lambda \in \rho(A)$,  
the resolvent set and resolvent of $A$, respectively.

Let us introduce the definition and some important results 
concerning the cosine and sine families.

\begin{definition}[See \cite{2}]
A one parameter family ${(C(t))}_{t\in \mathbb{R}}$ of bounded linear 
operators mapping the Banach space $X$ into itself is called 
a strongly continuous cosine family if, and only if:
\begin{enumerate}
\item $C(0)=I$,  where $I$ is the identity function of $X$;
\item $C(t+s)+C(t-s) =2C(t)C(s), \quad \forall (s,t) \in \mathbb{R}^2$;
\item $t \mapsto C(t)x$ is continuous in $\mathbb{R}$ for each $x$ in $X$.
\end{enumerate}
We define the sine family ${(S(t))}_{t\in \mathbb{R}}$ 
associated with the strongly continuous cosine family
${(C(t))}_{t \in \mathbb{R}}$ as follows:
$$ 
S(t)x=\int_0^t C(s)xds, \quad x \in X ,\quad t \in \mathbb{R}. 
$$ 
The infinitesimal generator of the cosine family ${(C(t))}_{t\in \mathbb{R}}$ 
is the operator $A$ defined by
$$  
Ax=\frac{d^2}{dt^2}C(0)x, \quad \forall x \in \mathcal{D}(A) 
:= \left \lbrace x \in X \ \ | \ \ C(t)x \in C^2(\mathbb{R}; X)\right\rbrace.
$$
The infinitesimal generator $A$ is a closed and densely defined operator in $X$.
\end{definition} 

\begin{proposition}[See \cite{2}]
\label{jjn}
Let us consider ${(S(t))}_{t \in \mathbb{R}}$  the sine family associated 
with the strongly continuous cosine family  ${(C(t))}_{t \in \mathbb{R}}$. 
The following assertions hold: 
\begin{itemize}
\item $\displaystyle \lim_{\varepsilon\longrightarrow 0} 
\dfrac{S(\varepsilon)x}{\varepsilon}=x$ for any $x\in X$;
\item if $ S(t)x \in  \mathcal{D} (A)$ for $x \in X$, 
then $\dfrac{d}{dt}C(t) x = AS(t)x$.
\end{itemize}
\end{proposition}

\begin{proposition}[See \cite{2}] 
\label{Bouaaaa}
Let $A$ be the infinitesimal generator of a strongly continuous 
cosine family of uniformly bounded linear operators  
${(C(t))}_{t\in {\mathbb{R}}}$ in $X$. We have:
$$ 
C(t)=\displaystyle\sum_{n=0}^{+\infty} \frac{A^nt^{2n}}{(2n)!}
\quad \text{for all} \quad t\in \mathbb{R}.
$$
\end{proposition} 

\begin{lemma}[See \cite{2}] 
\label{23} 
Let $A$ be the infinitesimal  generator of $(C(t))_{t\in \mathbb{R}}$. 
Then, for $\lambda\in \rho(A)$ such that $Re \lambda > 0$ 
and $\lambda^2 \in \rho(A)$, we have:
$$
\lambda  R(\lambda^2; A)x 
= \displaystyle\int_{0}^{+\infty} e^{-\lambda t}C(t)x dt, \quad x \in X,
$$ 
and
$$  
R(\lambda^2; A)x = \displaystyle\int_{0}^{+\infty} e^{-\lambda t}S(t)x dt, \quad x \in X.
$$
\end{lemma} 

For more details about  strongly continuous cosine and sine families, 
we refer the reader to \cite{Bouaaa,5}.

Next, we present some definitions and properties related 
to fractional integrals and derivatives.

\begin{definition}[See \cite{Almeidaa}]
Let $\alpha > 0$,  $\zeta$ be an integrable function defined on $J $ and
$ \psi \in C^{1}(J)$ be an increasing function such that 
$\psi^{\prime}(t)\neq 0,$ for all $t \in J$.
The left Riemann--Liouville fractional integral of a function $\zeta$ 
with respect to another function $\psi$ is defined by
$$ 
I_{0^+}^{\alpha,\psi}\zeta(t): = \frac{1}{\Gamma(\alpha)}
\int_0^t \psi^{\prime}(t){(\psi(t)-\psi(s))}^{\alpha-1}\zeta(s)ds,
$$
where $\Gamma(\alpha)=\displaystyle\int_{0}^{\infty}t^{\alpha-1}e^{-t}dt$ 
is the Euler Gamma function.	
\end{definition}

\begin{definition}[See \cite{Almeidaa}]
Let $\alpha > 0$,  $n = [\alpha] + 1$, and $ \zeta$, $ \psi \in C^n(J)$ 
be two functions such that $\psi$ is increasing and $\psi^{\prime}(t) \neq 0$  
for all $ t \in J$.  The left Caputo fractional derivative of a function 
$\zeta$ of order $\alpha$  with respect to another function $\psi$ is defined by
$$ 
\mathcal{^CD}^{\alpha ,\psi}_{0^{+}}\zeta(t)
:=I_{0^+}^{n-\alpha,\psi} {\left( \frac{1}{\psi^{\prime}(t)} \frac{d}{dt}\right)}^n\zeta(t). 
$$
\end{definition}

Let us denote 
$$ 
\zeta_{\psi}^{[k]}(t):= {\left( \frac{1}{\psi^{\prime}(t)} \frac{d}{dt}\right)}^k\zeta(t).
$$
The following proposition holds.

\begin{proposition}[See \cite{Almeidaa}]
\label{lalla}
Let $\zeta \in  C^n(J)$ and $ \alpha >0$. Then, 
$$ 
I_{0^+}^{\alpha,\psi} \left[ {\mathcal{^CD}^{\alpha ,\psi}_{0^{+}}\zeta(t)}\right] 
:= \zeta(t)+\displaystyle\sum_{k=0}^{n-1}\dfrac{\zeta_{\psi}^{[k]}(0^+)}{k!}(\psi(t)-\psi(0))^{k}. 
$$
\end{proposition}

Let us now recall the  definition and some propositions of the  
Laplace transform and the generalized Gronwall's inequality  
with respect to the function $\psi$.

\begin{definition}[See \cite{T.Abdeljawadd}]
Let  $ \zeta$, $ \psi $: $ [a,+\infty[\rightarrow \mathbb{R}$ 
be real valued functions such that $ \zeta$ is
continuous and $ \zeta >0$ on $ [0,+\infty [$. 
The generalized Laplace transform of $ \zeta$ is
defined by
$$ 
\mathcal{L}_{\psi}\left\lbrace \zeta (t)\right\rbrace (\lambda)
:= \int_a^{+\infty} e^{-\lambda(\psi(s)-\psi(a))}\zeta(s)\psi^{\prime}(s)ds
$$
for all values of $\lambda$.
\end{definition}

\begin{proposition}[See \cite{T.Abdeljawadd}]
\label{lokl}
Let $\alpha > 0$ and $ \zeta >0$ be a piecewise continuous function on each interval
$[0, t]$ of $\psi(t)$-exponential order, i.e., $\exists M,a,T>0$  
such that  $|\zeta(t)|< M e^{a\psi(t)}$ for all $t\geq T$.
Then, 
$$ 
\begin{aligned}
&1)\quad \mathcal{L}_{\psi}\left\lbrace  (1)\right\rbrace (\mathcal{G}): = \dfrac{1}{\mathcal{G}};\\
&2)\quad \mathcal{L}_{\psi}\left\lbrace I_{0^+}^{\alpha,\psi} \zeta (t)\right\rbrace (\mathcal{G})
: = \dfrac{\mathcal{L}_{\psi} \left\{\zeta(t)\right\}}{\mathcal{G}^\alpha};
\end{aligned}
$$
for all $\mathcal{G}\in\mathbb{C}$ such that $Re(\mathcal{G})\neq 0$.
\end{proposition}

\begin{definition}[See \cite{T.Abdeljawadd}]
\label{jijij}
Let $ \zeta$ and $\varphi$ be two functions that are continuous 
at each interval $[0,T]$. We define the generalized convolution 
of $ \zeta$ and $\varphi$  by
$$  
(\zeta \ast_{\psi}\varphi ) (t)
:=\int_a^t \zeta(s)\varphi \left( \psi^{-1}(\psi(t)
+\psi(a)-\psi(s)\right)  \psi^\prime(s) ds,  \quad \forall a \in J.        
$$
\end{definition}

\begin{proposition}[See \cite{T.Abdeljawadd}]
\label{jijijj}
Let $\delta$ and $ \kappa $ be two piecewise functions that 
are continuous at each interval $[0 ; T ]$ of $\psi(t)$-exponential order. 
Then,  
$$ 
\mathcal{L}_\psi \lbrace \delta \ast_\psi \kappa \rbrace 
= \mathcal{L}_\psi \lbrace \delta \rbrace \times  \mathcal{L}_\psi \lbrace \kappa \rbrace. 
$$
\end{proposition}

\begin{theorem}[Gronwall's inequality \cite{Adjabi}] 
\label{nnn}
Let $\varsigma$, $\sigma $ be two integrable functions and 
$\Lambda $ be a continuous function on $[a, b]$. Let 
$\Lambda \in C^1([a, b])$ be an increasing function 
such that $ \psi^{\prime}(t) \neq 0 $  
for all $t$ in $[a, b]$. Assume that
\begin{enumerate}
\item  $ \varsigma $ and $ \sigma$ are non-negative;
\item  $\Lambda $ is non-negative and non-decreasing.
\end{enumerate}
If 
$$ 
\varsigma(t) \leq  \sigma(t) + \Lambda(t)\int_a^t (\psi(t) 
- \psi(s))^{\alpha -1}\varsigma(s)\psi^{\prime}(s) ds,\quad \forall t \in [a, b],
$$
then
$$   
\varsigma(t) \leq \sigma(t) +    \int_a^t \sum_{j=1}^{+\infty}
\frac{\left[ \Lambda(t)\Gamma(\alpha)\right]^j }{\Gamma(j\alpha)}(\psi(t) 
- \psi(s))^{j\alpha -1}\sigma(s)\psi^{\prime}(s) ds,\quad \forall t \in [a, b].
$$
\end{theorem}

In the following, we present the definition of the two parameter Mittag-Leffler 
function, which is a useful function to obtain the expression of the mild solutions.

\begin{definition}[See \cite{6}]
The two parameter Mittag-Leffler function is defined as
$$
E_{\tau,\eta}(z)=\sum_{j=0}^{+\infty}\frac{z^j}{\Gamma(j\tau+\eta )}, 
\quad \tau,\ \eta >0, \quad z\in \mathbb{C}.
$$
\end{definition}  

For useful properties of the Mittag-Leffler function, 
we refer the reader to \cite{podlubny.1998,8}.


\section{Mild solution, resolvent operator and properties}
\label{sec:03}

In order to define the mild solution of the system \eqref{k}, 
we first consider Mainardi's Wright-type function $M_q$ (see  \cite{podlubny.1998}): 
$$
M_q(z)=\displaystyle\sum_{n=0}^{+\infty}  
\frac{(-z)^n}{n!\Gamma(1-q(n+1))},\quad z\in \mathbb{R}^+.
$$
The Mainardi's Wright-type function possesses many useful properties. 
We recall the following useful lemma.

\begin{lemma}[See \cite{Shuibo}]
\label{jjjjj}
For any $t > 0$, the Mainardi's Wright-type function satisfies:
\begin{enumerate}
\item $M_q(\theta) >0, \ \forall \theta\geq0$;
\item $ \displaystyle\int_{0}^{+\infty} \theta^n M_q(\theta) d\theta 
= \frac{\Gamma(1+n)}{\Gamma(1+nq)}, \quad   n > -1,\ 0<q<1$.
\end{enumerate} 
For additional properties of the Mainardi's Wright-type function, 
we refer the reader to \cite{Shuibo1}.
\end{lemma}

In addition, we consider the strongly continuous cosine family  
${(C(t))}_{t\in {\mathbb{R}}}$ such that  there exists a constant 
$\mathcal{K} \geq 1$ such that
\begin{equation}
\label{TT1}
\Vert C(t)\Vert_{_{L(X)}} \leq \mathcal{K},\quad  
\text{	for all $t\in \mathbb{R}$}.
\end{equation} 
Furthermore, for any $(t_1 ,\ t_2) 
\in \mathbb{R}\times\mathbb{R}$, we have
\begin{equation}
\label{TT2}
\Vert S(t_1)-S(t_2)\Vert_{_{L(X)}} \leq \mathcal{K} \vert t_2-t_1\vert.
\end{equation} 

For the rest of this paper,  we set 
$$
q := \displaystyle\frac{\alpha}{2}
$$ 
for $\alpha \in ]1,2]$. Hence, $q\in ]0,1]$.

\begin{lemma}
The system \eqref{k} is equivalent to the following integral equation:
\begin{equation}
\label{b}
\Theta(t)=\Theta_{0} +\left[ \frac{\psi(t)-\psi(0)}{\psi^\prime(0)}\right] 
\Theta_{1} + I_{0^+}^{\alpha,\psi}(A\Theta(t)+f(t,\Theta(t))),\ t\in J.
\end{equation}
\end{lemma}

\begin{proof}
$(\Rightarrow)$ \quad  Assume that  \eqref{k} is satisfied. If $t$ is in $J$, then 
\begin{equation}
\label{c}
I_{0^{+}}^{\alpha, \psi}(\mathcal{^CD}^{\alpha ,\psi}_{0^{+}} \Theta(t))
= I_{0^{+}}^{\alpha, \psi} \left( A\Theta(t)+f(t,\Theta(t))\right).
\end{equation}
On the other hand, by using Proposition~\ref{lalla}, we get
$$
\begin{aligned}
I_{0^{+}}^{\alpha, \psi}(\mathcal{^CD}^{\alpha ,\psi}_{0^{+}} \Theta(t))
&= \Theta(t)-  \Theta^{\left[ 0 \right]}_{\psi}(0) 
-\Theta^{\left[ 1 \right]}_{\psi}(0)\left( \psi(t)-\psi(0)\right) \\
&= \Theta{(t)}-\Theta_0 - \left[ \frac{\psi(t)-\psi(0)}{\psi^\prime(0)}\right] \Theta_{1}.
\end{aligned}
$$ 
Using the equation \eqref{c}, we obtain:
$$   
\Theta(t)=\Theta_{0} +\left[ \frac{\psi(t)-\psi(0)}{\psi^\prime(0)}\right] 
\Theta_{1} + I_{0^+}^{\alpha,\psi}\left( A\Theta(t)+f(t,\Theta(t))\right).
$$

$(\Leftarrow)$ \quad Suppose that equation \eqref{b} is satisfied.
Then, for $t=0$, we have:
$$
\Theta(0)=\Theta_0.
$$
Moreover, for all $t$ in $J$, we have:
\begin{equation*}
\begin{aligned}
&\dfrac{d}{dt} \Theta(t)
=\left[ \frac{\psi^\prime(t)}{\psi^\prime(0)}\right] \Theta_{1} 
+\dfrac{d}{dt}I_{0^{+}}^{\alpha,\psi} \left[  A\Theta(t)+f(t,\Theta(t)) \right] \\
& = \left[ \frac{\psi^\prime(t)}{\psi^\prime(0)}\right] \Theta_{1} 
+\dfrac{d}{dt}\left[  \dfrac{1}{\Gamma(\alpha)}\int_0^t \psi^\prime(\tau)(\psi(t)
-\psi(\tau))^{\alpha-1} \left[  A\Theta(\tau)+f(\tau,\Theta(\tau)) \right]d \tau \right] \\
&= \left[ \frac{\psi^\prime(t)}{\psi^\prime(0)}\right] \Theta_{1} 
+  \dfrac{1}{\Gamma(\alpha)}\int_0^t \dfrac{d}{dt}\left[  \psi^\prime(\tau)(\psi(t)
-\psi(\tau))^{\alpha-1} \left(  A\Theta(\tau)+f(\tau,\Theta(\tau)) \right) \right] d \tau \\ 
&\quad \quad + \dfrac{1}{\Gamma(\alpha)} \psi^\prime(t)(\psi(t)-\psi(t))^{\alpha-1} 
\left(  A\Theta(t)+f(t,\Theta(t)) \right)\dfrac{d}{dt}(t)  \\
&\quad \quad - \dfrac{1}{\Gamma(\alpha)} \psi^\prime(0)(\psi(t)-\psi(0))^{\alpha-1} 
\left(  A\Theta(0)+f(0,\Theta(0)) \right)\dfrac{d}{dt}(0)\\
&= \left[ \frac{\psi^\prime(t)}{\psi^\prime(0)}\right] \Theta_{1} 
+  \dfrac{\alpha-1}{\Gamma(\alpha)}\int_0^t   \psi^\prime(\tau) \psi^\prime(t) 
(\psi(t)-\psi(\tau))^{\alpha-2} \left(  A\Theta(\tau)+f(\tau,\Theta(\tau)) \right)  d \tau. 
\end{aligned}
\end{equation*}
Taking $t=0$, we get:
$$
\dfrac{d}{dt}\Theta(0)=\Theta_{1}.
$$
Moreover, 
\begin{equation*}
\begin{split}
\mathcal{^CD}^{\alpha ,\psi}_{0^{+}}\Theta(t)
&= \mathcal{^CD}^{\alpha ,\psi}_{0^{+}}\Theta_{0} 
+\mathcal{^CD}^{\alpha ,\psi}_{0^{+}}  
\left[ \frac{\psi(t)-\psi(0)}{\psi^\prime(0)}\right]
\Theta_{1} \\
&\qquad +\mathcal{^CD}^{\alpha ,\psi}_{0^{+}} I^{\alpha ,\psi}_{0^{+}} (A\Theta(t)+f(t,\Theta(t)))\\
&= \dfrac{\Theta_{1}}{\psi^\prime(0)}  
\left[ \mathcal{^CD}^{\alpha ,\psi}_{0^{+}}(\psi(t)-\psi(0))\right]  +A\Theta(t)+f(t,\Theta(t)).
\end{split}
\end{equation*}
According to \cite{Almeidaa}, 
$\mathcal{^CD}^{\alpha ,\psi}_{0^{+}}  \left[ \psi(t)-\psi(0)\right] =0$   
for all $ \alpha >1$. We obtain that
$$
\mathcal{^CD}^{\alpha ,\psi}_{0^{+}}\Theta(t)= A\Theta(t)+f(t,\Theta(t)),
$$
which completes the proof.\proofend
\end{proof}

\begin{theorem}
\label{bbb}
If the integral equation \eqref{b} holds, then:
$$
\Theta(t)= C^\psi_q(t,0) \Theta_{0} + R^\psi_q(t,0) \Theta_{1} 
+ \int_0^t (\psi(t)-\psi(s))^{q-1} P^\psi_q(t,s)f(s,\Theta(s))\psi^\prime(s)  ds,
$$
where
\begin{gather*}
C^\psi_q(t,0)x= \displaystyle\int_{0}^{+\infty}  
M_q(\theta)C((\psi(t)-\psi(0))^q\theta)xd\theta,\quad q=\displaystyle\frac{\alpha}{2},\  x\in X,\\
R^\psi_q(t,0)x= \displaystyle\int_{0}^{+\infty} C^\psi_q(s,0)\dfrac{\psi^\prime(s)}{\psi^\prime(0)}xds,\\
P^\psi_q(t,s)x=\displaystyle\int_{0}^{+\infty} q\theta M_q(\theta)S((\psi(t)-\psi(s))^q\theta)xd\theta.
\end{gather*}		
\end{theorem}

\begin{proof} 
Let $\lambda$ be such that $Re(\lambda) > 0 $ and $x\in X$. 
The generalized Laplace transform of the functions $\Theta$ and $f$ are, respectively, given by
$$
\vartheta_\psi(\lambda)
= \mathcal{L}_\psi(\Theta(t))(\lambda) 
= \int_0^{+\infty} e^{-\lambda(\psi(s)-\psi(0))}\Theta(s)\psi^{\prime}(s)ds,
$$
and
$$
\Upsilon_\psi(\lambda)= \mathcal{L}_\psi (f(t,\Theta(t)))(\lambda)
= \int_0^{+\infty} e^{-\lambda(\psi(s)-\psi(0))}f(s,\Theta(s))\psi^{\prime}(s)ds.
$$
Let us now take $\lambda$ such that $\lambda^\alpha \in \rho(A)$.  
Applying the  generalized Laplace transform to (\ref{b}), we get
$$
\vartheta_\psi(\lambda)=\frac{\Theta_{0}}{\lambda}
+\frac{\Theta_{1}}{\lambda^{2}\psi^\prime(0)}
+\frac{A\vartheta_\psi(\lambda)}{\lambda^{\alpha}}
+\frac{\Upsilon_\psi(\lambda)}{\lambda^{\alpha}}.
$$
Indeed, we have
$$   
\Theta(t)=\Theta_{0} +\left[ \frac{\psi(t)-\psi(0)}{\psi^\prime(0)}\right] 
\Theta_{1} + I_{0^+}^{\alpha,\psi}\left( A\Theta(t)+f(t,\Theta(t))\right)
$$
and	by applying  the generalized Laplace transforms, one obtains
\begin{multline*}
\mathcal{L}_\psi(\Theta(t))(\lambda)
= \mathcal{L}_\psi (\Theta_{0})(\lambda) + \mathcal{L}_\psi 
\left( \left[ \frac{\psi(t)-\psi(0)}{\psi^\prime(0)}\right] 
\Theta_{1}\right) (\lambda) \\
+  \mathcal{L}_\psi 
(I_{0^+}^{\alpha,\psi}\left( A\Theta(t)+f(t,\Theta(t))\right))(\lambda).
\end{multline*}  
Using Proposition~\ref{lokl}, we obtain that
\begin{gather*}
\mathcal{L}_\psi (\Theta_{0})(\lambda)=\dfrac{\Theta_{0}}{\lambda};\\
\mathcal{L}_\psi (I_{0^+}^{\alpha,\psi}\left( A\Theta(t)
+f(t,\Theta(t))\right))(\lambda)=\dfrac{1}{\lambda^\alpha}
A\mathcal{L}_\psi (\Theta(t))(\lambda) +\dfrac{1}{\lambda^\alpha}
\mathcal{L}_\psi (f(t,\Theta(t)))(\lambda);
\end{gather*}
and  
$$ 
\begin{aligned} 
\mathcal{L}_\psi \left( \left[ \frac{\psi(t)
-\psi(0)}{\psi^\prime(0)}\right] \Theta_{1}\right)(\lambda) 
&=\int_{0}^{+\infty} e^{-\lambda (\psi(s)-\psi(0))}\left[ 
\dfrac{\psi(s)-\psi(0)}{\psi^\prime(0)}\right] \psi^\prime(s)\Theta_{1}ds\\
&=\dfrac{\Theta_{1}}{{\psi^\prime(0)}} \int_{0}^{+\infty} 
\dfrac{d}{d \lambda}\left[ - e^{-\lambda (\psi(s)-\psi(0))}\right]  \psi^\prime(s)ds\\	
&=-\dfrac{\Theta_{1}}{{\psi^\prime(0)}} \dfrac{d}{d \lambda}  \left(  \dfrac{1}{\lambda} \right)\\
&= \dfrac{\Theta_{1}}{{\psi^\prime(0)}}   \dfrac{1}{\lambda^2}.
\end{aligned}
$$
Hence, 
\begin{equation*}
\vartheta_\psi(\lambda)=\frac{\Theta_{0}}{\lambda}+\frac{\Theta_{1}}{\lambda^{2}\psi^\prime(0)}
+\frac{A\vartheta_\psi(\lambda)}{\lambda^{\alpha}}+\frac{\Upsilon_\psi(\lambda)}{\lambda^{\alpha}}.
\end{equation*}
It follows that
\begin{equation} 
\label{24}
\vartheta_\psi(\lambda) =\lambda^{\alpha-1}\times (\lambda^\alpha I-A)^{-1} 
\Theta_{0}+\lambda^{\alpha-2}(\lambda^\alpha I-A)^{-1}
\dfrac{\Theta_{1}}{\psi^\prime(0)}+(\lambda^\alpha I-A)^{-1}\Upsilon_\psi(\lambda).
\end{equation}
By Lemma~\ref{23}, we can see that, for any $ x\in X$, we have
$$ 
\begin{aligned}
&\bullet \quad \lambda^{\alpha-1} (\lambda^\alpha I-A)^{-1}x 
= \lambda^{q-1}\lambda^{q}({(\lambda^{q})}^2 I-A)^{-1}x 
= \displaystyle\int_{0}^{+\infty} \lambda^{q-1} e^{-\lambda^{q} t}C(t)xdt,\\
&\bullet \quad (\lambda^\alpha I-A)^{-1}x =((\lambda^q)^{2}I-A)^{-1}x
=\displaystyle\int_{0}^{+\infty}e^{-\lambda^qt} S(t)xdt,	
\end{aligned}
$$
where $q=\dfrac{\alpha}{2}$. Let us consider $t=\tau^q$. 
We can write the above expressions as follows:
$$
\begin{aligned}
& \bullet \quad \lambda^{\alpha-1} (\lambda^\alpha I-A)^{-1}x 
= \displaystyle\int_0^{+\infty} 
q(\lambda\tau)^{q-1}e^{-{(\lambda \tau)}^q} C(\tau^q) x d\tau;\\
& \bullet \quad (\lambda^\alpha I-A)^{-1}x
= \displaystyle\int_0^{+\infty} q(\tau)^{q-1}e^{-{(\lambda \tau)}^q} S(\tau^q) x d\tau.
\end{aligned}
$$
If we take  $\tau = \psi(t)-\psi(0)$, then $ d\tau = \psi^\prime(t)dt$ 
and we obtain that
$$  
\begin{aligned}
\lambda^{\alpha-1}& (\lambda^\alpha I-A)^{-1}x\\ 
&= \int_0^{+\infty} q\left[ \lambda ( \psi(t)-\psi(0))\right]^{q-1}
e^{-{\left[ \lambda ( \psi(t)-\psi(0))\right] }^q} 
C(\left[  \psi(t)-\psi(0)\right] ^q)\psi^\prime(t) x dt\\
&= \int_0^{+\infty} \dfrac{-1}{\lambda}\dfrac{d}{dt} \left( 
e^{-{\left[ \lambda ( \psi(t)-\psi(0))\right] }^q} \right) 
C(\left[  \psi(t)-\psi(0)\right] ^q) x dt. 
\end{aligned}
$$
In addition, we have 
\begin{multline}
\quad (\lambda^\alpha I-A)^{-1}x 
= \int_0^{+\infty} 
q\left[ \psi(t)-\psi(0)\right] ^{q-1}e^{-{\left[ 
\lambda ( \psi(t)-\psi(0))\right] }^q}\\ 
\times S (\left[  \psi(t)-\psi(0)\right] ^q)\psi^\prime(t) x dt.
\end{multline}
Using the function 
$$ 
\phi_q(\theta)=\frac{q}{\theta^{q+1}}M_q(\theta^{-q}), \quad \forall\theta \in ]0,+\infty[ 
$$
defined in \cite{zhou.new}, and its Laplace transform given by
$$ 
\int_0^{+\infty} e^{-\lambda \theta}\phi_q(\theta)d\theta =e^{-\lambda^q},
$$
we obtain that
$$  
\begin{aligned}
\lambda^{\alpha-1}&({(\lambda^{\alpha})} I-A)^{-1}x \\
&=\int_0^{+\infty} 
\dfrac{-1}{\lambda}\dfrac{d}{dt}  \left( \int_0^{+\infty}   
e^{-{\left[ \lambda ( \psi(t)-\psi(0))\right]\theta }} 
\phi_q(\theta) d\theta \right) C(\left[  \psi(t)-\psi(0)\right] ^q) x dt\\
&=\int_0^{+\infty}  \int_0^{+\infty}\theta   e^{-{\left[ 
\lambda ( \psi(t)-\psi(0))\right]\theta }} \phi_q(\theta)   
C(\left[  \psi(t)-\psi(0)\right] ^q) x \psi^\prime(t)d\theta dt,	
\end{aligned}
$$
and 
$$  
\begin{aligned}
\quad (\lambda^\alpha I-A)^{-1}x 
&= \int_0^{+\infty} q\left[ \psi(t)-\psi(0)\right] ^{q-1}\\  
&\quad \times \int_0^{+\infty} e^{-{\left[ \lambda ( \psi(t)-\psi(0))\right] }\theta} 
\phi_q(\theta) d\theta S (\left[  \psi(t)-\psi(0)\right] ^q)  \psi^\prime(t) x dt\\
&= \int_0^{+\infty} \int_0^{+\infty}  q\left[ \psi(t)-\psi(0)\right] ^{q-1}  
e^{-{\left[ \lambda ( \psi(t)-\psi(0))\right] }\theta} \phi_q(\theta) \\
&\quad \times S (\left[  \psi(t)-\psi(0)\right] ^q)  \psi^\prime(t) x  d\theta dt.
\end{aligned}
$$
Set $\psi(s)=\psi(t)\theta +\psi(0)(1-\theta)$. 
Thus, $\psi^\prime(s)ds=\theta \psi^\prime(t)dt$. 
Then, we obtain: 
$$  
\begin{aligned}
\lambda^{\alpha-1}({\lambda^{\alpha}} I-A)^{-1}x 
&=\int_0^{+\infty}  \int_0^{+\infty}   e^{-{\left[ \lambda ( \psi(s)-\psi(0))\right] }} \phi_q(\theta) \\
&\quad \times C\left( \left[ \dfrac{ \psi(s)-\psi(0)}{\theta}\right] ^q\right)  x \psi^\prime(s)d\theta ds\\
&=\mathcal{L}_\psi \left(  \int_0^{+\infty} \phi_q(\theta)   
C\left( \left[ \dfrac{ \psi(t)-\psi(0)}{\theta}\right] ^q\right)  x d\theta \right)(\lambda),
\end{aligned}
$$
and 
$$  
\begin{aligned}
&(\lambda^\alpha I-A)^{-1}x 
=\int_0^{+\infty} \int_0^{+\infty}  q\dfrac{\left[ \psi(s)
-\psi(0)\right] ^{q-1}}{\theta^q}  e^{-{\left[ \lambda ( \psi(s)-\psi(0))\right] }} 
\phi_q(\theta) \\
&\qquad\qquad\qquad\qquad\qquad \times S \left( \left[  
\dfrac{\psi(s)-\psi(0)}{\theta}\right] ^q\right) \psi^\prime(s) x  d\theta ds\\
&\qquad =\mathcal{L}_\psi \left( \int_0^{+\infty}  
q\dfrac{\left[ \psi(t)-\psi(0)\right] ^{q-1}}{\theta^q}  \phi_q(\theta) 
S \left( \left[  \dfrac{\psi(t)-\psi(0)}{\theta}\right]^q\right)  
x  d\theta \right)(\lambda).
\end{aligned}
$$
By the substitution $l=\theta^{-q}$, we obtain:
$$  
\begin{aligned}
\lambda^{\alpha-1}({\lambda^{\alpha}} I-A)^{-1}x 
&=\mathcal{L}_\psi \left( \int_0^{+\infty} \dfrac{l^{-1-\frac{1}{q}}}{q}
\phi_q(l^{-\frac{1}{q}})   C\left( \left[ 
( \psi(t)-\psi(0))\right] ^q l\right)  x dl  \right)(\lambda)\\
&=\mathcal{L}_\psi \left( \left[ \int_0^{+\infty} M_q(l)   
C\left( \left[ ( \psi(t)-\psi(0))\right] ^q l\right)  x dl \right] \right) (\lambda)
\end{aligned}
$$
and 
$$  
\begin{aligned}
(&\lambda^\alpha I-A)^{-1}x\\ 
&= \mathcal{L}_\psi \left(  \int_0^{+\infty}  
ql \left[ \psi(t)-\psi(0)\right] ^{q-1}  \dfrac{l^{-1-\frac{1}{q}}}{q}
\phi_q(l^{-\frac{1}{q}})  S \left( \left[\psi(t)-\psi(0)\right] ^q l \right) x  dl  \right)(\lambda)\\
&= \mathcal{L}_\psi \left( \left[  \int_0^{+\infty}  ql \left[ \psi(t)-\psi(0)\right]^{q-1}  
M_q(l) S \left( \left[\psi(t)-\psi(0)\right] ^q l \right)   x  dl \right]\right)  (\lambda).
\end{aligned}
$$
This means that:
\begin{enumerate}
\item $\lambda^{\alpha-1}(\lambda^\alpha I-A)^{-1}x=\mathcal{L}_\psi 
\left( \left[ \displaystyle\int_0^{+\infty} M_q(l)   
C\left( \left[ ( \psi(t)-\psi(0))\right] ^q l\right)  x dl \right] \right) (\lambda)$;
\item $(\lambda^\alpha I-A)^{-1}x$ 
$$
=\mathcal{L}_\psi \left( \left[  
\displaystyle\int_0^{+\infty}  ql \left[ \psi(t)-\psi(0)\right] ^{q-1}  
M_q(l) S \left( \left[\psi(t)-\psi(0)\right] ^q l \right)   x  dl \right]\right)  ds.
$$
\end{enumerate}
Similarly, we obtain that
$$ 
\lambda^{\alpha-2}(\lambda^\alpha I-A)^{-1}x 
=\lambda^{-1}\times \mathcal{L}_\psi \left( \left[ \int_0^{+\infty} M_q(l)   
C\left( \left[ ( \psi(t)-\psi(0))\right] ^q l\right)  x dl \right] \right) (\lambda)
$$ 
and
\begin{multline*}
(\lambda^\alpha I-A)^{-1} \times \Upsilon_\psi(\lambda)\\
=\mathcal{L}_\psi \left( \left[  \displaystyle\int_0^{+\infty}  
ql \left[ \psi(t)-\psi(0)\right] ^{q-1}  M_q(l) S \left( 
\left[\psi(t)-\psi(0)\right] ^q l \right)    dl \right]\right) 
\times \Upsilon_\psi(\lambda).
\end{multline*} 
Since $\mathcal{L}_\psi(1)(\lambda)=\lambda^{-1}$, 
by using Proposition~\ref{jijijj} and Definition~\ref{jijij} we obtain that
\begin{equation*}
\begin{split}
\lambda^{\alpha-2} &\times(\lambda^\alpha I-A)^{-1}x \\
&=\mathcal{L}_\psi(1)
\times \mathcal{L}_\psi \left( \left[ \int_0^{+\infty} M_q(l)   
C\left( \left[ ( \psi(s)-\psi(0))\right] ^q l\right)  x dl \right] \right) (\lambda)\\
&=\mathcal{L}_\psi \bigg( \bigg[ 
\int_0^{t} \int_0^{+\infty}  M_q(l)   
C\left( \left[ \left( \psi\left(\psi^{-1}\left(\psi(t)+\psi(0)-\psi(\tau)\right) \right) 
-\psi(0)\right)\right] ^q l\right) \\
&\qquad\qquad\qquad 
\times x \psi^\prime(\tau)dl d\tau \bigg] \bigg) (\lambda)
\end{split}
\end{equation*}
and 
$$ 
\begin{aligned}
(&\lambda^\alpha I-A)^{-1} \times \Upsilon_\psi(\lambda)\\
&=\mathcal{L}_\psi \left[  \displaystyle \int_0^{t} f(\psi^{-1}\left(\psi(t)+\psi(0)
-\psi(\tau)\right) ,\Theta(\psi^{-1}\left(\psi(t)+\psi(0)-\psi(\tau))\right) )    \right. \\ 
&\quad \quad \quad \quad \left.\times \left[ \psi(\tau)-\psi(0)\right] ^{q-1}  
\displaystyle \int_0^{+\infty}  ql   M_q(l) 
S \left( \left[\psi(\tau)-\psi(0)\right] ^q l \right)  dl \psi^\prime(\tau) d\tau \right].  
\end{aligned}
$$
Taking $s=\psi^{-1}\left(\psi(t)+\psi(0)-\psi(\tau) \right)$, 
it follows that $\psi^\prime(s)ds=-\psi^\prime(\tau)d\tau$.
Then, we obtain:
\begin{multline*}
\lambda^{\alpha-2} \times(\lambda^\alpha I-A)^{-1}x\\ 
=\mathcal{L}_\psi \left( \left[ \int_0^{t} \int_0^{+\infty}  M_q(l)   
C \left( \left[ ( \psi(s) -\psi(0)) \right] ^q l \right)  
x  \psi^\prime(s)dl ds \right] \right) (\lambda);
\end{multline*}
\begin{multline*}
(\lambda^\alpha I-A)^{-1} \times \Upsilon_\psi(\lambda)
=\mathcal{L}_\psi \left(  \displaystyle \int_0^{t}  \left[ \psi(t)-\psi(s)\right] ^{q-1}  
f(s ,\Theta(s))  \psi^\prime(s) \right. \\ 
\left.\times \left[ \int_0^{+\infty}   
ql  M_q(l) S \left( \left[\psi(t)-\psi(s)\right]^q 
l \right) dl \right]   ds \right)(\lambda).  
\end{multline*}
Equation \eqref{24} becomes 
$$
\begin{aligned}
\vartheta_\psi(\lambda)
&= \mathcal{L}_\psi \left( \left[ \displaystyle\int_0^{+\infty} M_q(l)   
C\left( \left[ ( \psi(t)-\psi(0))\right] ^q l\right)  \Theta_0 dl \right] \right) (\lambda)\\
&\quad+\mathcal{L}_\psi \left( \left[ \int_0^{t} \int_0^{+\infty}  M_q(l)   
C \left( \left[ ( \psi(s) -\psi(0)) \right] ^q l \right)  
\dfrac{\Theta_1}{\psi^\prime(0)}  \psi^\prime(s)dl ds \right] \right) (\lambda)\\
&\quad+\mathcal{L}_\psi \bigg(  \displaystyle \int_0^{t}  \left[ \psi(t)-\psi(s)
\right]^{q-1} \left[  \int_0^{+\infty}   ql   M_q(l) S \left( \left[\psi(t)
-\psi(s)\right] ^q l \right) dl \right]   \\
&\qquad\qquad\qquad \times 
f(s ,\Theta(s))  \psi^\prime(s) ds \bigg)(\lambda)\\
&=\mathcal{L}_\psi \bigg( C^\psi_q(t,0) \Theta_{0} + K^\psi_q(t,0) \Theta_{1}\\ 
&\qquad \qquad \qquad
+ \int_0^t (\psi(t)-\psi(s))^{q-1} P^\psi_q(t,s)
f(s,\Theta(s))\psi^\prime(s)  ds\bigg)(\lambda).
\end{aligned}
$$
Hence, by using the uniqueness theorem of the generalized Laplace transform, we get
\begin{multline*}
\vartheta_\psi(\lambda)
=\mathcal{L}_\psi \bigg( C^\psi_q(t,0) \Theta_{0} 
+ R^\psi_q(t,0) \Theta_{1}\\ 
+ \int_0^t (\psi(t)-\psi(s))^{q-1} P^\psi_q(t,s)
f(s,\Theta(s))\psi^\prime(s)  ds\bigg)(\lambda),
\end{multline*}
where
\begin{equation*}
\begin{split}
C^\psi_q(t,0)
&= \displaystyle\int_{0}^{+\infty}  M_q(l)C((\psi(t)-\psi(0))^ql)dl;\\
R^\psi_q(t,0)
&= \displaystyle\int_{0}^{t} C^\psi_q(s,0)
\dfrac{\psi^\prime(s)}{\psi^\prime(0)}ds;\\
P^\psi_q(t,s)
&=\displaystyle\int_{0}^{+\infty} ql M_q(l)S((\psi(t)-\psi(s))^ql)dl.
\end{split}
\end{equation*}
So, by applying the inverse  generalized Laplace transform, we obtain that
\begin{multline*}
\Theta(t)=  C^\psi_q(t,0) \Theta_{0} + R^\psi_q(t,0) \Theta_{1}\\ 
+ \int_0^t (\psi(t)-\psi(s))^{q-1} P^\psi_q(t,s)f(s,\Theta(s))\psi^\prime(s)  ds.  
\end{multline*}
This completes the proof.\proofend
\end{proof}

Now we give the following definition of mild solution.

\begin{definition}
A function $\Theta \in C(J,X)$ is called a mild solution 
of system  \eqref{k} if it satisfies
$$
\Theta(t) =C^\psi_q(t,0) \Theta_{0} + R^\psi_q(t,0) \Theta_{1} 
+ \int_0^t (\psi(t)-\psi(s))^{q-1} P^\psi_q(t,s)f(s,\Theta(s))\psi^\prime(s)  ds.
$$
\end{definition}

The operators $C^\psi_q(t,0)$,  $R^\psi_q(t,0)$ 
and $P^\psi_q(t,s)$ satisfy the following properties.

\begin{proposition}
\label{aa}
For any  $ t \geq s > 0  $ and $ x \in X$, we have:
\begin{enumerate}
\item $\Vert C^\psi_q(t,0) x \Vert_{X} 
\leq  \mathcal{K} \Vert x \Vert_{X};$ \label{jjjjjj}
\item $ \Vert R^\psi_q(t,0) x \Vert_{X} 
\leq  \mathcal{K} \Vert x \Vert_{X} \dfrac{(\psi(t)-\psi(0)) }{\psi^\prime(0)};$
\item $ \Vert P^\psi_q(t,s)x \Vert_{X} 
\leq \dfrac{\mathcal{K}(\psi(t)-\psi(s))^{q}  }{\Gamma(2q)}  \Vert x \Vert_{X}$.
\end{enumerate}
\end{proposition}

\begin{proof}
Let $x$ be an element of $X $ and $t \geq s > 0$.

$1$. For the first inequality, we have: 
$$
\begin{aligned}
\Vert C^\psi_q(t,0) x \Vert_{X} 
&= \left\| \displaystyle\int_{0}^{+\infty}  
M_q(\theta)C((\psi(t)-\psi(0))^q\theta)xd\theta \right\|_{X}\\
&  \leq  \displaystyle\int_{0}^{+\infty}  
M_q(\theta)\left\| C((\psi(t)-\psi(0))^q\theta)
\right\|_{L(X)} \left\| x \right\|_{X}d\theta.
\end{aligned}
$$

Using \eqref{TT1} and Lemma~\ref{jjjjj}, we obtain:
$$
\begin{aligned}
\Vert C^\psi_q(t,0) x \Vert_{X} & \leq \mathcal{K} \Vert x \Vert_{X} .
\end{aligned}
$$

$2$. We can see that:
$$ 
\begin{aligned}
\left\| R^\psi_q(t,0) x \right\|_{X} 
&=  \left\|   \displaystyle\int_{0}^{t} 
C^\psi_q(s,0)x \dfrac{\psi^\prime(s)}{\psi^\prime(0)}ds \right\|_{X} \\
&\leq  \dfrac{\mathcal{K} \Vert x \Vert_{X} }{\psi^\prime(0)} 
\displaystyle\int_{0}^{t} \psi^\prime(s)ds  \\
&\leq \mathcal{K} \Vert x \Vert_{X} \dfrac{(\psi(t)-\psi(0)) }{\psi^\prime(0)}.
\end{aligned} 
$$

$3$. The third inequality can be obtained as follows:
$$ 
\begin{aligned}
\left\|P^\psi_q(t,s) x \right\|_X 
&= \left\| \int_{0}^{+\infty} q\theta M_q(\theta)S((\psi(t)
-\psi(s))^q\theta)x d\theta \right\|_X  \\ 
&\leq  \int_{0}^{+\infty} q\theta M_q(\theta) \left\|\int^{(\psi(t)
-\psi(s))^q\theta}_0 C(l) x dl  \right\|_Xd\theta. 
\end{aligned}
$$
By applying the condition \eqref{TT1}, we obtain:
$$ 
\begin{aligned}
\left\|P^\psi_q(t,s) x \right\|_X
&\leq \mathcal{K} q \left\| x\right\|_X  \vert (\psi(t)
-\psi(s))^q\vert  \int_{0}^{+\infty} \theta^2 M_q(\theta)d\theta \\
&\leq \mathcal{K} \dfrac{(\psi(t)-\psi(s))^q }{\Gamma(2q)}\left\| x\right\|_X.  
\end{aligned}
$$
Now, using Lemma~\ref{jjjjj}, 
$$ 
\begin{aligned}
\left\|P^\psi_q(t,s) x \right\|_X 
&\leq \mathcal{K} \dfrac{(\psi(t)-\psi(s))^q }{\Gamma(2q)}\left\| x\right\|_X.  
\end{aligned}
$$
This completes the proof.\proofend
\end{proof}

We give also the following proposition.

\begin{proposition}
\label{A}
The operators $C^\psi_q(t,0),  P^\psi_q(t,s)$ and $R^\psi_q(t,0)$
are strongly continuous, i.e., for every $u\in X$ and 
$0 \leq   s \leq  t_1 \leq t_2 \leq  T,$  we have: 
\begin{enumerate}
\item $ \Vert C^\psi_q(t_2,0)u-C^\psi_q(t_1,0)u \Vert_{{X}}
\longrightarrow 0 \  , \quad \mbox{whenever} \quad  t_1- t_2 \longrightarrow 0$;
\item $\Vert R^\psi_q(t_2,0)u- R^\psi_q(t_1,0) u \Vert_{{X}}\longrightarrow 0 
\  , \quad \mbox{whenever} \quad  t_1- t_2 \longrightarrow 0$;
\item $\Vert P^\psi_q(t_2,s)u-P^\psi_q(t_1,s)u \Vert_{{X}}
\longrightarrow 0 \  , \quad \mbox{whenever} \quad  t_1- t_2 \longrightarrow 0$.
\end{enumerate}
\end{proposition} 

\begin{proof}\ 
\begin{enumerate}
\item  For every $u \in X,$  ${(C(t))}_{t\in \mathbb{R}}$ 
is strongly continuous, i.e., for any $\varepsilon > 0$
and $  t_1,t_2 \in \mathbb{R},$  it follows that 
$$
\left\| C(t_2)u-C(t_1)u \right\|_{{X}} \longrightarrow0,
$$ 
whenever $t_1- t_2 \longrightarrow 0$.  
Therefore, for any $0 \leq    t_1 \leq t_2 \leq  T,$  we have:
\begin{equation*}
\begin{aligned}
&\left\| C^\psi_q(t_2,0)u-C^\psi_q(t_1,0)u \right\|_{{X}}\\ 
&\quad = \left\| \int_{0}^{\infty}  M_q(\theta)\left[ C\left( (\psi(t_1)
-\psi(0))^q\theta\right)  u-C\left( (\psi(t_2)-\psi(0))^q \theta\right)u \right]d\theta \right\|_X\\
&\quad\leq   \int_{0}^{\delta}  M_q(\theta)\left\| C\left( (\psi(t_1)
-\psi(0))^q \theta\right) u-C\left( (\psi(t_2)-\psi(0))^q \theta\right)  u\right\|_X  d\theta\\
&\qquad +  \int_{\delta}^{\infty}  M_q(\theta)\left\|  
C\left( (\psi(t_1)-\psi(0))^q \theta\right) u
-C\left((\psi(t_2)-\psi(0))^q \theta \right) u\right\|_X  d\theta.\\
\end{aligned}
\end{equation*}
Let us choose $\delta$ sufficiently large such that
$$  
\int_{\delta}^{\infty}  M_q(\theta)\left\|  C\left( (\psi(t_1)-\psi(0))^q 
\theta \right) u-C\left((\psi(t_2)-\psi(0))^q 
\theta \right) u\right\|_X  d\theta \leq \varepsilon. 
$$
This implies that
$$  
\int_{\delta}^{\infty}  M_q(\theta)\left\|  C\left( (\psi(t_1)-\psi(0))^q 
\theta\right) u-C\left( (\psi(t_2)-\psi(0))^q \theta \right) u\right\|_X  
d\theta  \longrightarrow 0
$$
as $\varepsilon \rightarrow 0$.
Furthermore, utilizing the condition 
\eqref{TT1} and Lemma~\ref{jjjjj}, we obtain that
\begin{multline*}
\int_0^{\delta}  M_q(\theta)\left\|  C\left( (\psi(t_1)-\psi(0))^q \theta\right) u
-C\left( (\psi(t_2)-\psi(0))^q \theta \right) u\right\|_X  d\theta \\
\leq 2 \mathcal{K}
\left\|u \right\|_X\int_0^{\delta}  M_q(\theta) d\theta 
\leq 2 \mathcal{K} \left\|u \right\|_X. 
\end{multline*}
Thus,
$$  
\int_0^{\delta}  M_q(\theta)\left\|  C\left( (\psi(t_1)-\psi(0))^q \theta\right) u
-C\left( (\psi(t_2)-\psi(0))^q \theta \right) u\right\|_X  d\theta   
\longrightarrow 0
$$
as $t_1- t_2 \rightarrow 0$. Consequently, 
$$
\Vert C^\psi_q(t_2,0)u-C^\psi_q(t_1,0)u \Vert_{{X}}\longrightarrow 0
$$ 
whenever $t_1- t_2 \longrightarrow 0$.

\item We can see that:
$$ 
\begin{aligned}
\left\| R^\psi_q(t_2,0)u-R^\psi_q(t_1,0)u \right\|_{{X}} 
&= \left\| \int_{0}^{t_1} C^\psi_q(s,0)u ds -\int_{0}^{t_2} C^\psi_q(s,0)u ds\right\|_X \\
&\leq  \left\| \int_{t_1}^{t_2} C^\psi_q(s,0)u ds\right\|_X \\
&\leq  \mathcal{K}  \left\| u \right\|_X (t_2-t_1) \longrightarrow 0,
\end{aligned}
$$
whenever $t_1- t_2 \rightarrow 0$.
\item For every $u\in X$ and $0 
\leq   s \leq  t_1 \leq t_2 \leq  T$, we have:
\begin{equation*}
\begin{split}
&\left\| P^\psi_q(t_2,s)u-P^\psi_q(t_1,s)u \right\|_{{X}}\\ 
& = \left\| \int_{0}^{\infty}q \theta  M_q(\theta)
\left[ S\left( (\psi(t_2)-\psi(s))^q \theta\right) u-S\left( (\psi(t_2)-\psi(s))^q 
\theta\right)u \right] d\theta\right\|_X\\
& \leq \int_{0}^{\infty} q \theta 
M_q(\theta)\left\| S\left( (\psi(t_1)-\psi(s))^q 
\theta\right) -S\left( (\psi(t_2)-\psi(s))^q 
\theta\right)\right\|_{L(X)}  \left\| u \right\|_{X}  d\theta.
\end{split}
\end{equation*}
Using condition \eqref{TT2} and Lemma~\ref{jjjjj}, we derive that
\begin{multline*}
\left\| P^\psi_q(t_2,s)u-P^\psi_q(t_1,s)u \right\|_{{X}}\\ 
 \leq \dfrac{\mathcal{K}}{\Gamma(2q)}\left\| u \right\|_{X} 
\vert(\psi(t_1)-\psi(s))^q - (\psi(t_2)-\psi(s))^q \vert.	
\end{multline*}
Consequently,  
$\Vert P^\psi_q(t_2,s)u-P^\psi_q(t_1,s)u 
\Vert_{{X}}\longrightarrow 0$ \mbox{whenever} 
$t_1- t_2 \longrightarrow 0$.
\end{enumerate}
This completes the proof.\proofend
\end{proof}

The operator $ P^\psi_q$ satisfies the following proposition.

\begin{proposition}
For every $u\in X$ and $0 \leq   s \leq  t_1 \leq t_2 \leq  T$ we have:
\begin{enumerate}
\item $\displaystyle\lim_{t_1\rightarrow s}(\psi(t_1)-\psi(s))^{q-1}P^\psi_q(t_1,s)u=0$;
\item $\Vert (\psi(t_2)-\psi(s))^{q-1}P^\psi_q(t_2,s)u-(\psi(t_1)-\psi(s))^{q-1}P^\psi_q(t_1,s)
u \Vert_{{X}}\longrightarrow 0$ \mbox{whenever} $t_1- t_2 \longrightarrow 0$.
\end{enumerate}
\end{proposition} 

\begin{proof}
\begin{enumerate}
\item Let us consider $u\in X$ and $0 \leq   s \leq  t_1 \leq  T$. 
By using Proposition~\ref{aa}, we have:
$$ 
\begin{aligned}
\left\|(\psi(t_1)-\psi(s))^{q-1}P^\psi_q(t_1,s)u \right\|_{X} 
\leq  \dfrac{(\psi(t_1)-\psi(s))^{2q-1} \mathcal{K}}{\Gamma(2q)}  \Vert u \Vert_{X},
\end{aligned}
$$
which implies that 
$ \displaystyle\lim_{t_1\rightarrow s}(\psi(t_1)-\psi(s))^{q-1}P^\psi_q(t_1,s)u=0$.

\item Let us consider $u\in X$, $0 \leq   s \leq  t_1 \leq t_2 \leq  T$, 
and  
$$ 
\varpi(s ,t_1, t_2,u )  = (\psi(t_2)-\psi(s))^{q-1}
P^\psi_q(t_2,s)u-(\psi(t_1)-\psi(s))^{q-1}P^\psi_q(t_1,s)u.
$$
Then,
$$ 
\begin{aligned}
\|\varpi&(s ,t_1, t_2,u ) \|_{X}\\
&\leq \left\|\left( (\psi(t_2)-\psi(s))^{q-1}-(\psi(t_1)
-\psi(s))^{q-1}\right) P^\psi_q(t_1,s)u \right\|_{X}\\
&\quad \quad +\left\| (\psi(t_2)-\psi(s))^{q-1} 
\left( P^\psi_q(t_2,s)u- P^\psi_q(t_1,s)u\right) \right\|_{X}\\
&\leq \vert (\psi(t_2)-\psi(s))^{q-1}\vert \left\| 
\left( P^\psi_q(t_2,s)u- P^\psi_q(t_1,s)u\right) \right\|_{X}\\
& \quad\quad +\left\vert \left( (\psi(t_2)-\psi(s))^{q-1}
-(\psi(t_1)-\psi(s))^{q-1}\right)\right\vert \left\| P^\psi_q(t_1,s)u \right\|_{X}.
\end{aligned}
$$
By using Propositions~\ref{aa} and \ref{A}, we get:
$$ 
\displaystyle\lim_{t_1- t_2 \longrightarrow 0}\Vert \varpi(s ,t_1, t_2,u ) \Vert_{{X}}= 0.
$$
\end{enumerate}
The proof is complete.\proofend
\end{proof}

\begin{proposition}
\label{ojg}
Assume that $C(t)$ and $S(t)$ are compact for every 
$t > 0$. Then the operators $C^\psi_q(t,0)$ and $R^\psi_q(t,0)$ are
compact for every $t >0$. 
\end{proposition}  

\begin{proof}
For each positive constant $\beta$, set 
$ B_\beta =\lbrace \Theta \in X \ | \ \Vert \Theta \Vert_X \leq \beta \rbrace $, 
which is a bounded, closed, and convex subset of $X$. We  prove that for any 
positive constant $\beta$ and $t > 0$, the sets
$$ 
W_1(t)=\lbrace C^\psi_q(t,0) \Theta_0, \quad   \Theta_0 \in    
B_\beta \rbrace \quad 
\text{and} \quad  W_2(t)=\lbrace R^\psi_q(t,0) \Theta_0, \quad   \Theta_0 \in    B_\beta \rbrace,  
$$
are relatively compact in $X$. Let us consider $ t>0$ be fixed, 
for any  $\delta > 0,$ and  $ 0<\varepsilon \leq t$.  
Let $\Theta_0 \in B_\beta$. We define the subsets
\begin{equation*}
\begin{split}
W_1^{\delta,\varepsilon}(t)
&=\left\lbrace  C(\varepsilon^q \delta)\int_\delta^{+\infty}M_q(\theta)
C((\psi(t)-\psi(0))^q\theta-\varepsilon^q \delta)\Theta_0 d\theta
\right\rbrace,\\
W_2^{\delta,\varepsilon}(t)
&=\left\lbrace  C(\varepsilon^q \delta)
\int_0^{t-\varepsilon}\int_\delta^{+\infty}M_q(\theta)C((\psi(s)
-\psi(0))^q\theta-\varepsilon^q \delta)\Theta_0 d\theta 
\dfrac{\psi^\prime(s)}{\psi^\prime(0)} ds\right\rbrace, 
\end{split}
\end{equation*}
and, for all $\Theta_0 \in    B_\beta$, we denote 
\begin{equation*}
C_q^{\psi,\delta,\varepsilon}(t,0)\Theta_0
=  C(\varepsilon^q \delta)\int_\delta^{+\infty}M_q(\theta)C((\psi(t)
-\psi(0))^q\theta-\varepsilon^q \delta)\Theta_0 d\theta,
\end{equation*}
and
\begin{multline*}
R_q^{\psi,\delta,\varepsilon}(t,0)\Theta_0\\
=  C(\varepsilon^q \delta)\int_0^{t-\varepsilon}\int_\delta^{+\infty}
M_q(\theta)C((\psi(s)-\psi(0))^q\theta-\varepsilon^q \delta)\Theta_0 
d\theta \dfrac{\psi^\prime(s)}{\psi^\prime(0)} ds.
\end{multline*}
For $\theta \in \, ]0,+\infty[$, it follows from
the uniform convergence of Mainardi's Wright 
type function and the uniform boundedness 
of the cosine  family \eqref{TT1}, that
$$
\begin{aligned}
\left\|C_q^{\psi,\delta,\varepsilon}(t,0)\Theta_0 \right\|_{X} 
&\leq \mathcal{K} \int_\delta^{+\infty}M_q(\theta)\left\| C((\psi(t)-\psi(0))^q\theta
-\varepsilon^q \delta)\Theta_0\right\|_{X} d\theta\\
&\leq \mathcal{K}^2 \left\| \Theta_0\right\|_{X}\int_0^{+\infty}M_q(\theta) d\theta\\
&\leq \mathcal{K}^2 \left\| \Theta_0 \right\|_{X},
\end{aligned}
$$
and 
$$
\begin{aligned}
\left\|R_q^{\psi,\delta,\varepsilon}(t,0)\Theta_0 \right\|_{X}
&\leq \mathcal{K}^2 \left\| \Theta_0\right\|_{X} \int_0^{t} 
\dfrac{\psi^\prime(s)}{\psi^\prime(0)} \int_0^{+\infty}  M_q(\theta) d\theta ds\\
&\leq \mathcal{K}^2 \left\| \Theta_0\right\|_{X}\left[ 
\dfrac{\psi(t)-\psi(0)}{\psi^\prime(0)}\right]\\
&\leq \mathcal{K}^2 \left\| \Theta_0\right\|_{X}\left[ 
\dfrac{\psi(T)-\psi(0)}{\psi^\prime(0)}\right].
\end{aligned}
$$
Hence, by the compactness of $C(t)$,  we see that  
$ W_1^{\delta,\varepsilon}(t)$ and $ W_2^{\delta,\varepsilon}(t)$ 
are relatively compact in $X$, for all $t >0$. Furthermore,
$$
\begin{aligned}
&\left\|C_q^{\psi,\delta,\varepsilon}(t,0)\Theta_0
-C_q^{\psi}(t,0)\Theta_0 \right\|_{X} \leq \left\| 
\int_0^{\delta}M_q(\theta)C((\psi(t)-\psi(0))^q\theta)\Theta_0 d\theta\right\|_{X}\\
&\quad+\left\|C(\varepsilon^q \delta) \int_\delta^{+\infty}
M_q(\theta)C((\psi(t)-\psi(0))^q\theta-\varepsilon^q \delta)\Theta_0 d\theta\right.\\
&\left.\qquad\qquad
-\int_\delta^{+\infty}M_q(\theta)C((\psi(t)-\psi(0))^q\theta)
\Theta_0 d\theta\right\|_{X}\\ 
&=:I_1+I_2,
\end{aligned}
$$
and 
$$
\begin{aligned}
&\left\|R_q^{\psi,\delta,\varepsilon}(t,0)
\Theta_0-R_q^{\psi}(t,0)\Theta_0 \right\|_{X}\\
&\quad \leq 
\left\|\int_{t-\varepsilon}^{t} \int_0^{+\infty}M_q(\theta)
C((\psi(s)-\psi(0))^q\theta )\Theta_0 d\theta \dfrac{\psi^\prime(s)}{\psi^\prime(0)} ds \right\|_{X}\\
&\qquad+\left\| \int_0^{t-\varepsilon} \int_\delta^{+\infty}  
M_q(\theta)\left[  C(\varepsilon^q \delta) 
C((\psi(s)-\psi(0))^q\theta-\varepsilon^q\delta)\right.\right.\\
&\left.\left.\qquad\qquad
-C((\psi(s)-\psi(0))^q\theta)\right] 
\Theta_0 d\theta\dfrac{\psi^\prime(s)}{\psi^\prime(0)} ds \right\|_{X}\\ 
&\qquad+\left\|\int_{0}^{t-\varepsilon} \int_0^{\delta}
M_q(\theta) C((\psi(s)-\psi(0))^q\theta )
\Theta_0 d\theta\dfrac{\psi^\prime(s)}{\psi^\prime(0)} ds \right\|_{X}\\
&\quad =:I_3+I_4+I_5.
\end{aligned}
$$
Using the uniform boundedness of the cosine family, we get 
$$ 
I_1\leq \mathcal{K} \Vert \Theta_0 \Vert_X \int_0^{\delta} M_q(\theta) d\theta
\longrightarrow 0 \  , \quad as \quad  \delta \longrightarrow 0.
$$ 
In addition,
$$
\begin{aligned}
I_2 &\leq \Bigg\| \int_\delta^{+\infty}M_q(\theta)\left[ C(\varepsilon^q \delta)
C((\psi(t)-\psi(0))^q\theta-\varepsilon^q \delta)\right.\\ 
&\left. \qquad\qquad
- C((\psi(t)-\psi(0))^q \theta) \right] \Theta_0 d\theta\Bigg\|_{X}\\
&\leq \Bigg\| \int_\delta^{+\infty}M_q(\theta) C(\varepsilon^q \delta)
\left[C((\psi(t)-\psi(0))^q\theta-\varepsilon^q \delta)\right.\\ 
&\left. \qquad\qquad
- C((\psi(t)-\psi(0))^q\theta) \right] \Theta_0 d\theta\Bigg\|_{X}\\ 
&\quad+\left\| \int_\delta^{+\infty}M_q(\theta) C((\psi(t)-\psi(0))^q
\theta)\left[C(\varepsilon^q \delta) - I \right] \Theta_0 d\theta\right\|_{X}.
\end{aligned}
$$
Using condition \eqref{TT2}, we derive 
$$
\begin{aligned}
I_2&\leq \mathcal{K} \int_0^{+\infty}M_q(\theta) \left\|\left[C((\psi(t)-\psi(0))^q\theta
-\varepsilon^q \delta) - C((\psi(t)-\psi(0))^q\theta) \right] \Theta_0\right\|_{X} d\theta\\ 
&\quad+ \mathcal{K}  \int_0^{+\infty}M_q(\theta)   d\theta.
\end{aligned}
$$

Using the strong continuity of the cosine family, 
we get $I_2 \longrightarrow 0$ as $\delta \longrightarrow 0$, 
which implies that $I_4 \longrightarrow 0$ as $\delta \longrightarrow 0$.
Moreover,
$$ 
\begin{aligned}
I_3 &\leq \mathcal{K} \left[ \dfrac{\psi(t)
-\psi(t-\varepsilon)}{\psi^\prime(0)}\right]\Vert \Theta_0 \Vert_X   
\longrightarrow 0 \  , \quad as \quad  \varepsilon \longrightarrow 0;\\
I_5 &\leq  \mathcal{K} \left[ \dfrac{\psi(t)-\psi(0)}{\psi^\prime(0)}\right]\Vert 
\Theta_0 \Vert_X \int_0^\delta M_q(\theta)d\theta  
\longrightarrow 0 \  , \quad as \quad  \delta \longrightarrow 0.
\end{aligned}
$$
Hence, $W_1^{\delta,\varepsilon}(t)$ 
and $ W_2^{\delta,\varepsilon}(t)$ are relatively compact sets arbitrarily 
close to the sets $ W_1(t)$ and $ W_2(t)$ for every $t > 0$. 
Thus,  the sets $ W_1(t)$ and $ W_2(t)$ are relatively compact 
in $X$ for every $t > 0$.\proofend
\end{proof}

We now prove a relation between the operators $P_q^\psi$ and $C_q^\psi$.

\begin{proposition}
If Theorem~\ref{bbb} holds, then for every $t > 0$ we have: 
\begin{enumerate}
\item $ (\psi(t)-\psi(0))^{q-1}P_q^{\psi}(t,0)x
=I_{0^+}^{2q-1,\psi}C_q^{\psi}(t,0)x,$ \quad for $x\in X$;

\item $ \dfrac{d}{dt} C_q^{\psi}(t,0)x
=\psi^\prime(t)(\psi(t)-\psi(0))^{q-1}AP_q^{\psi}(t,0)x,$ \quad for $x\in X$.
\end{enumerate}
\end{proposition}

\begin{proof}
From the proof of Theorem~\ref{bbb}, for any $x \in X$, we know that for $Re\lambda > 0,$
$$
\lambda^{2q-1}(\lambda^{2q}I-A)^{-1}x
=\mathcal{L}_\psi(C^\psi_q(t,0)x)(\lambda) 
$$
and 
$$
(\lambda^{2q}I-A)^{-1}x
=\mathcal{L}_\psi((\psi(t)-\psi(0))^{q-1} P^\psi_q(t,0)x)(\lambda).
$$
Since
$$
\mathcal{L}_\psi\left( \dfrac{(\psi(t)
-\psi(0))^{2q-2}}{\Gamma(2q-1)}\right) (\lambda)=\lambda^{1-2q},
$$
we get
$$
\begin{aligned}
&(\lambda^{2q}I-A)^{-1}x
=\lambda^{1-2q} \mathcal{L}_\psi(C^\psi_q(t,0)x)(\lambda)\\
&=\mathcal{L}_\psi\left( \dfrac{(\psi(t)-\psi(0))^{2q-2}}{\Gamma(2q-1)}\right) 
(\lambda)\times\mathcal{L}_\psi(C^\psi_q(t,0)x)(\lambda)\\
&=\mathcal{L}_\psi\left( \int_0^t\dfrac{(\psi(\tau)-\psi(0))^{2q-2}}{\Gamma(2q-1)} 
C^\psi_q(\psi^{-1}(\psi(t)+\psi(0)-\psi(\tau)),0)x\psi^\prime(\tau) d\tau \right) (\lambda).
\end{aligned}
$$
Using the variable $s= \psi^{-1}(\psi(t)+\psi(0)-\psi(\tau))$, we deduce that
$$ 
(\lambda^{2q}I-A)^{-1}x
= \mathcal{L}_\psi\left( \int_0^t\dfrac{(\psi(t)-\psi(s))^{2q-2}}{\Gamma(2q-1)} 
C^\psi_q(s,0)x\psi^\prime(s) ds \right) (\lambda),
$$
which implies
\begin{multline*}
\mathcal{L}_\psi((\psi(t)-\psi(0))^{q-1} P^\psi_q(t,0)x)(\lambda)\\
= \mathcal{L}_\psi\left( \int_0^t\dfrac{(\psi(t)-\psi(s))^{2q-2}}{\Gamma(2q-1)} 
C^\psi_q(s,0)x\psi^\prime(s) ds \right) (\lambda).
\end{multline*}
Now, by applying the inverse Laplace transform, we get 
$$ 
(\psi(t)-\psi(0))^{q-1} P^\psi_q(t,0)x =  \dfrac{1}{\Gamma(2q-1)} 
\int_0^t (\psi(t)-\psi(s))^{2q-2} C^\psi_q(s,0)x\psi^\prime(s) ds.
$$
Moreover, from Proposition~\ref{jjn} we have:
$$
\begin{aligned}
\dfrac{d}{dt} C^\psi_q(t,0) x
&=\displaystyle\int_{0}^{+\infty}  M_q(\theta)\dfrac{d}{dt} 
C((\psi(t)-\psi(0))^q\theta)x d\theta\\
&=\displaystyle\int_{0}^{+\infty} q\theta \psi^\prime(t) (\psi(t)-\psi(0))^{q-1}  
M_q(\theta)A S((\psi(t)-\psi(0))^q\theta)x d\theta\\
&=\displaystyle  \psi^\prime(t) (\psi(t)-\psi(0))^{q-1} A \int_{0}^{+\infty} 
q\theta  M_q(\theta) S((\psi(t)-\psi(0))^q\theta)x d\theta\\
&=\displaystyle  \psi^\prime(t) (\psi(t)-\psi(0))^{q-1} A P^\psi_q(t,0)x.
\end{aligned}
$$
This completes the proof.\proofend
\end{proof}


\section{Existence and uniqueness of mild solution}
\label{sec:04}

In this section, we present two results where we prove the existence 
and uniqueness of the mild solution for system \eqref{k}. 
For the first result, we make the following assumptions:
\begin{description}
\item[(H1)] The function $ f : J \times X \rightarrow  X $ 
is continuous and there exists a constant $ N_0 >0 $ such that
$$ 
\Vert f(t,x)-f(t,y) \Vert_X \leq N_0 \Vert x-y\Vert_X ,\  
\forall t \in J,\ \forall (x,y)  \in   X\times X .
$$ 
\item[(H2)] The inequality 
$\dfrac{\mathcal{K} N_0 (\psi(T)-\psi(0))^{2q}}{\Gamma(2q+1)}<1$ holds.
\end{description}

\begin{theorem}
\label{jjjj}
If (H1) and (H2) are satisfied, then 
system  \eqref{k} has a unique mild solution.
\end{theorem}

\begin{proof}
Let us consider the operator $ \Psi :\  C(J;X)\rightarrow C(J;X) $ defined by
\begin{multline*}
(\Psi \Theta)(t)=C^\psi_q(t,0) \Theta_{0} + R^\psi_q(t,0) \Theta_{1}\\ 
+ \int_0^t (\psi(t)-\psi(s))^{q-1} P^\psi_q(t,s)f(s,\Theta(s))\psi^\prime(s)  ds, \ t\in J.
\end{multline*}
Let us consider $\Theta ,\Phi \ \in \ C(J;X) $. Then,
$$
\begin{aligned}
\Vert&\Psi\Theta-\Psi\Phi \Vert
=\displaystyle\sup_{t\in J} \Vert (\Psi\Theta)(t)-(\Psi\Phi)(t)\Vert_X\\
&=\displaystyle\sup_{t\in J} \left\| \int_0^t (\psi(t)-\psi(s))^{q-1} 
P^\psi_q(t,s)[ f(s,\Theta(s))-f(s,\Phi(s))]\psi^\prime(s)ds \right\|_X \\
&\leq \displaystyle\sup_{t\in J}  \int_0^t \vert (\psi(t)-\psi(s))^{q-1} 
\vert \Vert P^\psi_q(t,s)\Vert_{_{L(X)}} 
\Vert f(s,\Theta(s)-f(s,\Phi(s)) \Vert_X \psi^\prime(s) ds.
\end{aligned}
$$
Using (H1) and Proposition~\ref{aa}, we get
$$
\begin{aligned}
\Vert \Psi\Theta-\Psi\Phi \Vert
&\leq \displaystyle\sup_{t\in J} \frac{\mathcal{K} N_0}{\Gamma(2q)} 
\int_0^t  (\psi(t)-\psi(s))^{2q-1}\Vert \Theta(s)-\Phi(s) \Vert_X \psi^\prime(s) ds\\
&\leq  \frac{\mathcal{K} N_0 (\psi(t)-\psi(0))^{2q}}{\Gamma(2q+1)}  
\displaystyle\sup_{s\in J} \Vert \Theta(s)-\Phi(s) \Vert_X\\
&\leq  \frac{\mathcal{K} N_0 (\psi(T)-\psi(0))^{2q}}{\Gamma(2q+1)}   
\Vert \Theta(\cdot)-\Phi(\cdot) \Vert.
\end{aligned}
$$
Now by using the hypothesis (H2) we obtain that
$$        
\Vert \Psi\Theta-\Psi\Phi \Vert\leq  \Vert \Theta(\cdot)-\Phi(\cdot) \Vert. 
$$
We deduce that $\Psi$ is a contraction mapping. Hence, 
by Banach's fixed point theorem, $\Psi$ has a unique fixed point 
in $C(J;X)$, $\Theta (t)=\Psi\Theta(t) $. Thus, it follows that the system   
\eqref{k} has a unique mild solution on $C(J;X)$.\proofend
\end{proof}

Before giving the second result, we introduce the following assumptions:
\begin{description}
\item[(H3)] The function $ f : J \times X \rightarrow  X $ 
is continuous and for all $\beta >0$ there exists 
a constant $ N_\beta >0 $ such that
$$ 
\Vert f(t,x(t))-f(t,y(t)) \Vert_X \leq N_\beta \Vert x(t)-y(t)\Vert_X ,
\  \forall t  \in J, \ \forall (x  ,  y)  \in  B_\beta\times B_\beta,
$$ 
where $B_\beta =\lbrace \Theta \in C(J;X) \ | \ \Vert \Theta \Vert \leq \beta \rbrace$.

\item[(H4)] The operators $C(t)$ and $ S(t)$  are compact for every $ t > 0 $.

\item[(H5)] For any $ \beta > 0 $, there exists a function 
$h_\beta \in  L^{\infty}(J; X)$ such that
$$ 
\displaystyle \sup_{\Vert y(t) \Vert\leq \beta} 
\Vert f(t,y(t)) \Vert_X \leq h_\beta(t) ,\ \forall  t\in J, 
$$
for all $y$ in $B_p$. Moreover, there exists a constant $\mathcal{H} > 0$ 
such that 
$$
\displaystyle \lim_{\beta \rightarrow \infty} \displaystyle\sup_{t\in J} 
\frac{\Vert h_\beta(t) \Vert_{L^{\infty}(J;X)}}{\beta} = \mathcal{H}.
$$
\end{description}

\begin{theorem}
\label{jjj}
Assume that the conditions (H3)--(H5) hold. 
If  
$$
\frac{ \mathcal{H} \mathcal{K} (\psi(T)-\psi(0))^{2q} }{ \Gamma(2q+1)} < 1,
$$
then the system \eqref{k} has a unique mild solution.
\end{theorem}

\begin{proof} 
For any $ \beta > 0$, let us consider  
$B_\beta =\lbrace \Theta \in C(J;X) \ | \ \Vert \Theta \Vert \leq \beta \rbrace $ 
a  subset of $C(J;X)$. Consider the operator 
$ \Psi : C(J;X)\rightarrow C(J;X) $ defined by
\begin{multline*}
(\Psi \Theta)(t)=C^\psi_q(t,0) \Theta_{0} + R^\psi_q(t,0) \Theta_{1}\\ 
+ \int_0^t (\psi(t)-\psi(s))^{q-1} P^\psi_q(t,s)f(s,\Theta(s))
\psi^\prime(s)  ds, \ t\in J.
\end{multline*}
\begin{description}
\item[Step 1.] 
We show that $ \Psi (B_\beta) \subset B_\beta$.
Assume that for each $ \beta > 0$ there exists   
$\Theta \ \in \ B_\beta $ and $ t \in  ]0, T] $ 
such that $ \Vert (\Psi \Theta)(t) \Vert_X >\beta $. Then, 
$$
\begin{aligned}
\beta &< \Vert (\Psi \Theta)(t) \Vert_X \\
&= \Big\| C^\psi_q(t,0) \Theta_{0} + R^\psi_q(t,0) \Theta_{1}\\ 
& \qquad + \int_0^t (\psi(t)-\psi(s))^{q-1} 
P^\psi_q(t,s)f(s,\Theta(s))\psi^\prime(s)  ds \Big\|_X \\
& \quad \leq \Vert C^\psi_q(t,0) \Theta_{0} \Vert_X
+\Vert R^\psi_q(t,0)\Theta_1\Vert_X \\
&\qquad +  \int_0^t  (\psi(t)-\psi(s))^{q-1}  \Vert P^\psi_q(t,s) 
\Vert_{L(X)} \Vert f(s,\Theta(s)) \Vert_X \psi^\prime(s) ds.
\end{aligned}
$$                                           
Using (H5) and Proposition~\ref{aa}, we get:
$$ 
\beta \leq  \mathcal{K} \Vert \Theta_0\Vert_X
+\frac{\mathcal{K} (\psi(T)-\psi(0))^{2q}}{\psi^\prime(0)\Gamma(2q)}
\Vert \Theta_1\Vert_X +\frac{\mathcal{K} \Vert h_\beta(t) 
\Vert_{L^{\infty}} (\psi(T)-\psi(0))^{2q} }{\Gamma(2q+1)}. 
$$
Dividing both sides by $ \beta $ and taking the sup limit  
as $ \beta \rightarrow +\infty $, we obtain:
$$ 
1< \frac{ \mathcal{H} \mathcal{K} (\psi(T)-\psi(0))^{2q} }{ \Gamma(2q+1)}  ,
$$
which is a contradiction with Theorem~\ref{jjj}. 
Then, $\Psi(B_\beta) \subset B_\beta $.

\item[Step 2.] We prove that $ \Psi (B_\beta)$ is equicontinuous. 
For$ \ 0 <s\leq  t_2  < t_1\leq T $, we have
$$
\begin{aligned}
&\Vert \Psi\Theta(t_1)-\Psi\Theta(t_2) \Vert_X \\
&\leq \Vert  
(C^\psi_q(t_1,0)-C^\psi_q(t_2,0))\Theta_0\Vert_X
+\Vert( R^\psi_q(t_1,0)-R^\psi_q(t_2,0))\Theta_1\Vert_X \\
&\quad +\Bigg\| \int_0^{t_2} [(\psi(t_1)-\psi(s))^{q-1}P^\psi_q(t_1,s)
-(\psi(t_2)-\psi(s))^{q-1} P^\psi_q(t_2,s)]\\
& \qquad\qquad \times f(s,\Theta(s)) \psi^\prime(s) ds \Bigg\|_X \\
& \quad +\left\| \int_{t_2}^{t_1} (\psi(t_1)-\psi(s))^{q-1}
P^\psi_q(t_1,s)f(s,\Theta(s))\psi^\prime(s) ds \right\|_X \\
&\leq  \Vert C^\psi_q(t_1,0)-C^\psi_q(t_2,0) \Vert_{L(X)} \Vert \Theta_0 
\Vert_X+\Vert R^\psi_q(t_1,0)-R^\psi_q(t_2,0)\Vert_{L(X)} \Vert \Theta_1\Vert_X \\
&\quad+ \int_{t_2}^{t_1}  (\psi(t_1)-\psi(s))^{q-1} 
\Vert P^\psi_q(t_1,s) \Vert_{L(X)} 
\Vert f(s,\Theta(s)) \Vert_X ds \\
&\quad+ \int_0^{t_2} \left[ \psi(t_1)-\psi(s))^{q-1} -(\psi(t_2)-\psi(s))^{q-1} \right]
\Vert P^\psi_q(t_1,s)\Vert_{L(X)} \\
&\qquad\qquad\quad \times 
\Vert f(s,\Theta(s))\Vert_X \psi^\prime(s) ds\\
&\quad+ \int_0^{t_2}  (\psi(t_2)-\psi(s))^{q-1} \Vert  P^\psi_q(t_1,s)-P^\psi_q(t_2,s) 
\Vert_{L(X)}  \Vert f(s,\Theta(s))\Vert_X \psi^\prime(s) ds\\
&:=Q_1+Q_2+Q_3+Q_4+Q_5.
\end{aligned}
$$
By Proposition~\ref{A}, it is clear that 
$Q_1, Q_2\rightarrow 0$ as $t_1 \rightarrow t_2$. We obtain that
$$ 
Q_3 \leq \frac{\mathcal{K} \Vert h_\beta 
\Vert_{L^{\infty}}}{\Gamma(2q+1)} (\psi(t_1)-\psi(t_2))^{2q} 
\longrightarrow 0, \quad as \quad  t_1 \longrightarrow t_2,                   
$$
and 
\begin{multline*}
Q_4\leq \left[ (\psi(T)-\psi(0))^{q}\right] \dfrac{2\mathcal{K} \Vert h_\beta 
\Vert_{L^{\infty}}}{\Gamma(2q+1)}\\
\times \left[ (\psi(t_1)-\psi(0))^{q}
-(\psi(t_2)-\psi(0))^{q}-(\psi(t_1)-\psi(t_2))^{q}\right]. 
\end{multline*}
Hence, $ Q_4\longrightarrow 0 $ as $ t_1 \longrightarrow t_2  $.
For $ Q_5 $, taking $ \varepsilon \in  ]0, t_2[ $ small enough,  we have:
$$
\begin{aligned}
Q_5 &\leq  \int_0^{t_2-\varepsilon}  (\psi(t_2)-\psi(s))^{q-1} \Vert  
P^\psi_q(t_1,s)-P^\psi_q(t_2,s) \Vert_{L(X)} \\
&\qquad\qquad \times \Vert f(s,\Theta(s))\Vert_X \psi^\prime(s) ds \\ 
& \quad +\int_{t_2-\varepsilon}^{t_2}  (\psi(t_2)-\psi(s))^{q-1} \Vert  
P^\psi_q(t_1,s)-P^\psi_q(t_2,s) \Vert_{L(X)} \\
&\qquad\qquad \times \Vert f(s,\Theta(s))\Vert_X \psi^\prime(s) ds\\
&\leq \frac{ \Vert h_\beta \Vert_{L^{\infty}}}{q} \sup_{0<s<t_2-\varepsilon}\Vert  
P^\psi_q(t_1,s)-P^\psi_q(t_2,s) \Vert_{L(X)} \\
&\qquad \times \left[ (\psi(t_2)
-\psi(0))^{q}-(\psi(t_2)-\psi(t_2-\varepsilon))^{q}\right]\\
&\qquad + \frac{2 \mathcal{K} \Vert h_\beta \Vert_{L^{\infty}}}{\Gamma(2q+1)} 
\left[ (\psi(t_2)-\psi(t_2-\varepsilon))^{q}\right].
\end{aligned}
$$
Consequently, $Q_5 \rightarrow 0$ as ${t_2 \rightarrow t_1}$ and $ \varepsilon \rightarrow 0$.
Thus, $ \Psi(B_\beta)$ is equicontinuous.

\item[Step 3.] We prove that, for any $ t\in J $, 
$ G(t) =\lbrace (\Psi\Theta)(t) \ |\ \Theta \in B_\beta\rbrace$ is relatively compact in $X$.
Let us consider $ t\in J,$  $\delta > 0,$  $ 0<\varepsilon \leq t$, 
and    $\Theta \in B_\beta$. We define the following operator: 
\begin{equation*}
\begin{aligned}
&(\Psi_{{\varepsilon,\delta}} \Theta)(t)= C^\psi_q(t,0) \Theta_{0} + R^\psi_q(t,0) \Theta_{1}\\
&\quad+\frac{S(\varepsilon^q\delta)}{\varepsilon^q\delta} \int_{0}^{t-\varepsilon} 
\int_{\delta}^{+\infty} (\psi(t)-\psi(s))^{q-1} q\theta M_q(\theta)
S((\psi(t)-\psi(s))^{q}\theta-\varepsilon^q\delta) \\
&\qquad\qquad\qquad\qquad\qquad\qquad 
\times f(s,\Theta(s))\psi^\prime(s) d\theta ds.
\end{aligned}
\end{equation*}
For $\theta \in \, ]0,+\infty[$,
by the uniform convergence of Mainardi's Wright type function   
and the uniform boundedness of the cosine and sine families,
$$
\begin{aligned}
& \Bigg\|  \int_{0}^{t-\varepsilon} \int_{\delta}^{+\infty} (\psi(t)-\psi(s))^{q-1} 
q\theta M_q(\theta)S((\psi(t)-\psi(s))^{q}\theta-\varepsilon^q\delta)\\ 
&\qquad\qquad\qquad \times f(s,\Theta(s))\psi^\prime(s) d\theta ds \Bigg\|_X \\
&\quad  \leq  \int_{0}^{t-\varepsilon} \int_{\delta}^{+\infty} 
(\psi(t)-\psi(s))^{q-1} q\theta M_q(\theta)\left\|S((\psi(t)-\psi(s))^{q}
\theta-\varepsilon^q\delta)\right\|_{L(X)} \\
&\qquad\qquad\qquad \times \left\| f(s,\Theta(s))  \right\|_X d\theta \psi^\prime(s)ds.
\end{aligned}
$$
Using condition~\eqref{TT1} and assumption (H5), we obtain: 
$$
\begin{aligned}
& \Bigg\|  \int_{0}^{t-\varepsilon} \int_{\delta}^{+\infty} (\psi(t)-\psi(s))^{q-1} 
q\theta M_q(\theta)S((\psi(t)-\psi(s))^{q}\theta-\varepsilon^q\delta)\\ 
&\qquad\qquad\qquad \times f(s,\Theta(s))\psi^\prime(s) d\theta ds \Bigg\|_X \\
&\  \leq q \Vert h_\beta \Vert_{L^{\infty}} \mathcal{K} \int_{0}^{t} 
\int_{0}^{+\infty} (\psi(t)-\psi(s))^{q-1} \theta M_q(\theta) ( (\psi(t)
-\psi(s))^{q }\theta+t^q\theta )d\theta\\
&\qquad\qquad\qquad\qquad\qquad \times \psi^\prime(s) ds\\
&\  \leq q \Vert h_\beta \Vert_{L^{\infty}} \mathcal{K} \int_{0}^{t} 
\left[ (\psi(t)-\psi(s))^{2q-1}+ (\psi(t)-\psi(s))^{q-1}t^q)\right] 
\psi^\prime(s)ds  \\
&\qquad \times \int_{0}^{+\infty}  \theta^2 M_q(\theta)  d\theta\\
&\ \leq \dfrac{2 \mathcal{K} \Vert h_\beta \Vert_{L^{\infty}}  
\left[ (\psi(T)-\psi(0))^{2q}+T^q(\psi(T)-\psi(0))^{q}\right] }{\Gamma(2q+1)}.
\end{aligned}
$$
Hence, by Proposition~\ref{ojg} and the compactness of  $ S(t)$, 
we see that the set 
$$ 
G_{\varepsilon,\delta}(t) 
=\lbrace (\Psi_{{\varepsilon,\delta}}
\Theta)(t) \ | \ \Theta \in B_\beta\rbrace
$$ 
is relatively compact in $X$, for all $t >0 $. 

Furthermore,
\begin{equation*}
\begin{aligned}
&\| (\Psi \Theta)(t)-(\Psi_{{\varepsilon,\delta}} \Theta)(t)\|_X \\
&\leq \Bigg\|  \int_{0}^{t} \int_{0}^{\delta} (\psi(t)-\psi(s))^{q-1} q\theta 
M_q(\theta)S((\psi(t)-\psi(s))^{q}\theta) \\
&\qquad\qquad\qquad \times f(s,\Theta(s)) d\theta \psi^\prime(s) 
ds \Bigg\|_X\\
&+\Bigg\|  \int_{t-\varepsilon}^{t} 
\int_{\delta}^{+\infty} (\psi(t)-\psi(s))^{q-1} q\theta 
M_q(\theta)S((\psi(t)-\psi(s))^{q}\theta) \\
&\qquad\qquad\qquad \times f(s,\Theta(s)) d\theta \psi^\prime(s) 
ds \Bigg\|_X\\
& + \left\| \int_{0}^{t-\varepsilon} 
\int_{\delta}^{+\infty} (\psi(t)-\psi(s))^{q-1} q\theta M_q(\theta) 
f(s,\Theta(s))\right. \\
& \ \ \left. 
\times \left(S((\psi(t)-\psi(s))^{q}\theta)
-\frac{S(\varepsilon^q\delta)}{\varepsilon^q\delta}S((\psi(t)-\psi(s))^{q}
\theta-\varepsilon^q\delta) \right)  d\theta \psi^\prime(s) ds \right\|_X\\
& =:O_{4_1}+O_{4_2}+O_{4_3}.
\end{aligned} 
\end{equation*}
By using (H5) and Lemma~\ref{jjjjj}, we obtain that
\begin{equation*}
\begin{aligned}
O_{4_1} &\leq \dfrac{ \mathcal{K} (\psi(t)-\psi(0))^{2q} \Vert 
h_\beta \Vert_{L^{\infty}}}{2} \int_{0}^{\delta}  \theta^2 M_q(\theta) d\theta, \\
O_{4_2}&\leq \dfrac{ \mathcal{K} (\psi(t)-\psi(t-\varepsilon))^{2q} 
\Vert h_\beta \Vert_{L^{\infty}}}{\Gamma(2q+1)}, \\
O_{4_3}  &\leq \Vert h_\beta \Vert_{L^{\infty}}  \int_{0}^{t} 
\int_{0}^{+\infty} (\psi(t)-\psi(s))^{q-1} q\theta M_q(\theta) \\
&\ \  \times \left\| S((\psi(t)-\psi(s))^{q}\theta) 
- \frac{S(\varepsilon^q\delta)}{\varepsilon^q\delta}
S((\psi(t)-\psi(s))^{q}\theta-\varepsilon^q\delta)  
\right\|_{L(X)}  d\theta ds. 
\end{aligned}
\end{equation*}
We get that $O_{4_1} \longrightarrow 0$ as $\delta\longrightarrow 0$  and  
$O_{4_2} \longrightarrow 0$ as $\varepsilon \longrightarrow 0$.
Moreover, 
\begin{equation*}
\begin{aligned}
& \int_{0}^{+\infty} (\psi(t)-\psi(s))^{q-1} q\theta M_q(\theta)\\ 
&\quad \times \left\| S((\psi(t)-\psi(s))^{q}\theta)- \frac{S(\varepsilon^q
\delta)}{\varepsilon^q\delta}S((\psi(t)-\psi(s))^{q}\theta
-\varepsilon^q\delta)  \right\|_{L(X)}\psi^\prime(s)  d\theta \\
&\quad \quad \quad \quad \quad 
\leq \dfrac{ (\psi(t)-\psi(s))^{2q-1}(\mathcal{K} 
+ \mathcal{K}^2)+(\psi(t)-\psi(s))^{q-1}t^q\mathcal{K}^2}{\Gamma(2q)} \psi^\prime(s).
\end{aligned}
\end{equation*}
Thus,
\begin{equation*}
\begin{split}
&\int_{0}^{t} \int_{0}^{+\infty} (\psi(t)-\psi(s))^{q-1} q\theta 
M_q(\theta) \\
&\times \left\| S((\psi(t)-\psi(s))^{q}\theta)
- \frac{S(\varepsilon^q\delta)}{\varepsilon^q\delta}S((\psi(t)
-\psi(s))^{q}\theta-\varepsilon^q\delta)  \right\|_{L(X)}  d\theta \psi^\prime(s)ds
\end{split}
\end{equation*}
is uniformly convergent. In addition, 
\begin{equation*}
\begin{aligned}
&  \left\| S((\psi(t)-\psi(s))^{q}\theta)
- \frac{S(\varepsilon^q\delta)}{\varepsilon^q\delta}S((\psi(t)
-\psi(s))^{q}\theta-\varepsilon^q\delta)  \right\|_{L(X)}\\
&\leq \left\| S((\psi(t)-\psi(s))^{q}\theta) \left( 
\frac{S(\varepsilon^q\delta)}{\varepsilon^q\delta}  
- I \right)  \right\|_{L(X)} \\
&\quad \quad 
+\left\|  \frac{S(\varepsilon^q\delta)}{\varepsilon^q\delta}\left( 
S((\psi(t)-\psi(s))^{q}\theta)-\varepsilon^q\delta) 
- S((\psi(t)-\psi(s))^{q}\theta) \right)  \right\|_{L(X)}\\
&\leq \mathcal{K} (\psi(t)-\psi(0))^{q}\theta  \left\|  
\frac{S(\varepsilon^q\delta)}{\varepsilon^q\delta}  - I \right\|_{L(X)}\\ 
&\quad \quad  + \mathcal{K} \left\|   S(((\psi(t)-\psi(s))^{q}\theta)
-\varepsilon^q\delta) - S((\psi(t)-\psi(s))^{q}\theta)\right\|_{L(X)}.
\end{aligned}
\end{equation*}
Therefore, using Proposition~\ref{jjn} and the strong 
continuity of the sine family, we get that  
$$
\left\| S((\psi(t)-\psi(s))^{q}\theta)- \frac{S(\varepsilon^q\delta)}{
\varepsilon^q\delta}S((\psi(t)-\psi(s))^{q}\theta-\varepsilon^q\delta)  
\right\|_{L(X)} \longrightarrow 0
$$
as $\delta\rightarrow 0$. Thus, 
\begin{multline*}
\int_{0}^{t} \int_{0}^{+\infty} (\psi(t)-\psi(s))^{q-1} q\theta M_q(\theta)\\ 
\times \left\| S((\psi(t)-\psi(s))^{q}\theta)- \frac{S(\varepsilon^q\delta)}{\varepsilon^q\delta}
S((\psi(t)-\psi(s))^{q}\theta-\varepsilon^q\delta)  \right\|_{L(X)}  d\theta \psi^\prime(s)ds
\end{multline*}
converges to $0$ where $\delta\rightarrow 0$. Consequently, 
we have  $O_{4_3} \rightarrow 0$  as $\delta\rightarrow 0$. 
\end{description} 
Finally, we obtain that
\begin{equation}
\left\| (\Psi \Theta)(t)-(\Psi_{{\varepsilon,\delta}} \Theta)(t)
\right\|_X \longrightarrow 0
\end{equation}
when $\varepsilon, \delta \rightarrow 0$.
Therefore, $G_{\varepsilon,\delta}(t)$ is a relatively compact set arbitrarily close 
to the set $G(t)$ for $t>0$. Hence, $G(t)$ is relatively compact.
Using the Arzela--Ascoli theorem and steps $2$ and $3$, we conclude that 
$ \Psi(B_\beta)$ is relatively compact. Furthermore, Schauder's fixed point theorem,  
step $1$, and the relative compactness of $ \Psi(B_\beta)$, allow us to deduce that 
$\Psi$ has a fixed point. Let us now show the uniqueness of that fixed point.
	
\begin{description}
\item[Step 4.] We show the uniqueness of the fixed point of $\Psi$.
Let $\Theta $ and $ \Phi$ be two fixed points of $\Psi$:
$$
\begin{aligned}
\Vert &\Theta(t)-\Phi(t) \Vert_{X} 
=\Vert (\Psi\Theta)(t)-(\Psi\Phi)(t)\Vert_X \\
&= \left\| \int_0^t (\psi(t)-\psi(s))^{q-1}P^\psi_q(t,s) 
[ f(s,\Theta(s))-f(s,\Phi(s))] \psi^\prime(s)  ds \right\|_X \\
&\leq   \int_0^t  (\psi(t)-\psi(s))^{q-1} \Vert P^\psi_q(t,s)\Vert_{L(X)} 
\Vert f(s,\Theta(s)-f(s,\Phi(s)) \Vert_X \psi^\prime(s) ds.
\end{aligned}
$$
From (H3) and Proposition~\ref{aa}, we get
$$
\begin{aligned}
\Vert \Theta(t)-\Phi(t) \Vert_{X} 
&\leq  \frac{\mathcal{K} N_0}{\Gamma(2q)} \int_0^t  (\psi(t)
-\psi(s))^{2q-1}\Vert \Theta(s)-\Phi(s) \Vert_X \psi^\prime(s) ds.
\end{aligned}
$$
By using Gr\"{o}nwall's inequality \eqref{nnn}, we obtain that
$$
\Vert \Theta(t)- \Phi(t) \Vert_{X} =0 \  , \forall  t\in J,
$$
which implies that $ \Theta \equiv \Phi $.
\end{description}
Therefore, the operator $\Psi$ has a unique fixed point, 
which is also the unique mild solution of system \eqref{k}.\proofend
\end{proof}


\section{Particular case}
\label{sec:05}

In this section we give the explicit expression of the mild solution 
for system  \eqref{k} for a particular dynamic $A$. For that we use
a series expansion involving the Mittag-Leffler function. 

Let $\Omega$ be an open and bounded set of $\mathbb{R}^n$, $n\geq 1$. 
We denote by $\partial\Omega$ its boundary, which is supposed to be 
sufficiently smooth. The space $X$ is chosen to be the space of all 
square integrable functions $L^2(\Omega)$, which is a Hilbert space. 
Let $A$ be the second order elliptic 
differential operator given by
\begin{equation}
A\Theta(x,t) = \displaystyle \sum_{i,j=1}^{n}
\frac{\partial}{\partial x_i}\left[ a_{ij}(x)
\frac{\partial\Theta(x,t)}{\partial x_j}\right] +a_0(x)
\Theta(x,t) ,\ \forall x\in \Omega, \ \forall t \in J,
\end{equation}
where the coefficients $a_{ij}$ are in $\ L^{\infty}(\Omega) $ 
for all $  1 \leq  i, j \leq n $, and $a_0$ is also in 
$ L^{\infty}(\Omega)$ such that
\begin{equation}
\label{hhh}
\begin{cases}
a_{i,j}=a_{j,i}, \quad      1 \leq  i, j \leq n ,\\ 
\exists \gamma > 0,\ \forall \iota \in \mathbb{R}^n  \  ,
\displaystyle \sum_{i,j=1}^{n}  a_{ij}(x)\iota_i\iota_j 
\geq \gamma \vert \iota \vert^2.
\end{cases}
\end{equation}
This means that $A$ is symmetric and $-A$ is uniformly elliptic. 
In this case, it is well known that $-A$ has a set of eigenvalues 
$(\lambda_k)_{k\geq1}$ such that
$$ 
0< \lambda_1< \lambda_2<\lambda_3<\lambda_4
<\cdots<\lambda_k<\lambda_{k+1}<\cdots
\rightarrow +\infty. 
$$
Each eigenvalue $\lambda_k$ corresponds to $r_k$ eigenfunctions 
$$
\lbrace\varphi_{k j}\rbrace_{1\leq j \leq r_k}, 
$$
where $r_k \in \mathbb{N}\setminus \{0\}$ is the multiplicity 
of $\lambda_k$ such that 
$$
A\varphi_{kj} = \lambda_k\varphi_{kj}
$$ 
and $\varphi_{kj} \in \mathcal{D}(A)$ for all 
$k \in \mathbb{N}^* $, $1\leq j \leq r_k$.  
Furthermore, the set 
$$
\left\{\varphi_{k j}\right\}_{{\substack{k\geq1\\1\leq r_k\leq k}}}
$$ 
constitutes an orthonormal basis of $X$. For more details see \cite{Floridia}.

The operator $A$ is the infinitesimal generator of a strongly continuous cosine 
family of uniformly bounded linear operators ${(C(t))}_{t\in {\mathbb{R}}}$ in $X$. 
The system  \eqref{k} with Dirichlet boundary conditions can be written as
\begin{equation}
\label{a.par}
\begin{cases}
\mathcal{^CD}^{\alpha ,\psi}_{0^{+}}\Theta(t)
=A \Theta(t)+f(t,\Theta(x,t)), \  
& \text{ in }\  {\Omega\times J},\\		
\Theta(\xi, t)=0, \ & \text{ on }\ \partial\Omega \times J,\\
{\Theta(x,0)=\Theta_{0}(x)} 
\  ;\  \dfrac{\partial}{\partial t}\Theta(x,0)
=\Theta_{1}(x), & \text{ in }  \ \Omega,
\end{cases}
\end{equation}
and it possesses one, and only one, mild solution provided 
that (H1)--(H2) or (H3)--(H5) are satisfied. 
Moreover, it has the following form:
\begin{equation}
\begin{aligned}
\Theta(x,t)
&=\displaystyle  \sum_{k=1}^{+\infty} \sum_{j=1}^{r_{k}} E_{2q,1}\left( (\psi(t)-\psi(0))^{2q} 
\lambda_k\right) \left\langle \Theta_0 , \varphi_{ kj} \right\rangle_{X} \varphi_{ kj}(x)\\
&+ \left[ \frac{(\psi(t)-\psi(0))}{\psi^\prime(0)}\right]  
\sum_{k=1}^{+\infty} \sum_{j=1}^{r_{k}} 
\left\langle \Theta_{1} , \varphi_{kj} \right\rangle_{X} 
\varphi_{kj}(x) E_{2q,2}\left(  (\psi(t)-\psi(0))^{2q} \lambda_k \right)\\
&+\displaystyle \sum_{k=1}^{+\infty}\sum_{j=1}^{r_k} 
\int_0^t  (\psi(t)-\psi(s))^{2q-1}  
E_{2q,2q}\left(\lambda_{k}
(\psi(t)-\psi(s))^{2q}\right) \\
&\qquad\qquad\qquad \times \left\langle f(s,\Theta(\cdot,s)),
\varphi_{kj} \right\rangle_{X} \psi^\prime(s) ds  \varphi_{kj}(x),
\end{aligned}
\end{equation}
where $ \left\langle \cdot , \cdot \right\rangle_{X} $ is the inner product in $X$.
Indeed, the system \eqref{a.par} has one and only one mild solution of the following form:
\begin{equation}
\label{ag}
\Theta(t)= C^\psi_q(t,0) \Theta_{0} + R^\psi_q(t,0) \Theta_{1} 
+ \int_0^t (\psi(t)-\psi(s))^{q-1} P^\psi_q(t,s)f(s,\Theta(s))\psi^\prime(s)  ds.
\end{equation}
From \eqref{Bouaaaa}, we can write the cosine family  
${(C(t))}_{t\in {\mathbb{R}}}$ as
$$  
C(t)\Theta_0(x) =\displaystyle \sum_{n=0}^{+\infty} \sum_{k=1}^{+\infty} 
\sum_{j=1}^{r_{k}}  \frac{t^{2n} \lambda_k^n}{(2n)!} \left\langle 
\Theta_0 , \varphi_{ kj} \right\rangle_{X} \varphi_{ kj}(x).
$$
Thus, 
$$
\begin{aligned}
C^\psi_q(t,0) &\Theta_{0}(x) \\ 
&=  \displaystyle\int_{0}^{\infty} M_q({\theta}) 
C{((\psi(t)-\psi(0))^q\theta)}\Theta_{0}(x) d\theta\\
&=  \int_{0}^{\infty} M_q({\theta}) \displaystyle 
\sum_{n=0}^{+\infty} \sum_{k=1}^{+\infty} \sum_{j=1}^{r_{k}}  
\frac{\left[ (\psi(t)-\psi(0))^q\theta\right]^{2n} 
\lambda_k^n}{2n!} \left\langle \Theta_0 , 
\varphi_{ kj} \right\rangle_{X} \varphi_{ kj}(x) d\theta\\
&= \displaystyle  \sum_{k=1}^{+\infty} 
\sum_{j=1}^{r_{k}} E_{2q,1}\left( (\psi(t)-\psi(0))^{2q} 
\lambda_k\right) \left\langle \Theta_0 , \varphi_{ kj} 
\right\rangle_{X} \varphi_{ kj}(x) ,
\end{aligned}
$$
$$
\begin{aligned}
R&^\psi_q(t,0)\Theta_{0}(x)\\  
&= \sum_{k=1}^{+\infty} \sum_{j=1}^{r_{k}}\left\langle 
\Theta_{0} , \varphi_{ kj}\right\rangle_{X} \varphi_{ kj}(x)
\int_0^t \frac{\psi^\prime(s)}{\psi^\prime(0)}  
E_{2q,1}\left( (\psi(s)-\psi(0))^{2q} \lambda_k\right) ds\\
&= \left[ \frac{(\psi(t)-\psi(0))}{\psi^\prime(0)}\right]  
\sum_{k=1}^{+\infty} \sum_{j=1}^{r_{k}} \left\langle 
\Theta_{0} , \varphi_{kj} \right\rangle_{X} \varphi_{kj}(x) 
\sum_{n=0}^{+\infty} \frac{\left[ 
(\psi(t)-\psi(0))^{2q} \lambda_k\right]^n}{\Gamma(2qn+2)}\\
&=\left[ \frac{(\psi(t)-\psi(0))}{\psi^\prime(0)}\right]  
\sum_{k=1}^{+\infty} \sum_{j=1}^{r_{k}} \left\langle \Theta_{0}, 
\varphi_{kj} \right\rangle_{X} \varphi_{kj}(x) 
E_{2q,2}\left(  (\psi(t)-\psi(0))^{2q} \lambda_k \right),
\end{aligned}
$$
and 
$$
\begin{aligned}
P&^\psi_{q}(t,\tau)\Theta_{1}(x)\\  
&= \displaystyle\int_{0}^{\infty} q \theta M_q({\theta}) 
S{((\psi(t)-\psi(\tau))^q\theta)}\Theta_{1}(x) d\theta \\
&=  \int_{0}^{\infty} q \theta M_q({\theta}) 
\int_{0}^{(\psi(t)-\psi(\tau))^q\theta} C(s)\Theta_{1}(x) ds d\theta \\
& =  \displaystyle\sum_{k=1}^{+\infty} \sum_{j=1}^{r_{k}} 
\sum_{n=0}^{+\infty} \left\langle \Theta_{1}, \varphi_{ kj} 
\right\rangle_{X} \varphi_{ kj}(x)  \int_{0}^{\infty} q \theta 
M_q{(\theta)} \frac{\lambda_k^n((\psi(t)
-\psi(\tau))^q\theta)^{2n+1}}{2n!(2n+1)} d\theta\\
&= (\psi(t)-\psi(\tau))^{q} \sum_{k=1}^{+\infty} \sum_{j=1}^{r_{k}}  
E_{2q,2q}\left(\lambda_{k} (\psi(t)-\psi(\tau))^{2q}\right) 
\varphi_{ kj}(x) \left\langle \Theta_{1} , \varphi_{j k} \right\rangle_{X}.
\end{aligned}
$$
From equation \eqref{ag}, we get that
\begin{equation*}
\begin{aligned}
\Theta(x,t)
&=\displaystyle  \sum_{k=1}^{+\infty} \sum_{j=1}^{r_{k}} 
E_{2q,1}\left( (\psi(t)-\psi(0))^{2q} \lambda_k\right) 
\left\langle \Theta_0 , \varphi_{ kj} \right\rangle_{X} \varphi_{ kj}(x)\\
&+ \left[ \frac{(\psi(t)-\psi(0))}{\psi^\prime(0)}\right] 
\sum_{k=1}^{+\infty} \sum_{j=1}^{r_{k}} \left\langle \Theta_{1}, 
\varphi_{kj} \right\rangle_{X} \varphi_{kj}(x) 
E_{2q,2}\left(  (\psi(t)-\psi(0))^{2q} \lambda_k \right)\\
&+\displaystyle \sum_{k=1}^{+\infty}\sum_{j=1}^{r_k} 
\int_0^t  (\psi(t)-\psi(s))^{2q-1}  E_{2q,2q}\left(\lambda_{k}
(\psi(t)-\psi(s))^{2q}\right) \\
&\qquad\qquad\qquad\qquad \times 
\left\langle f(s,\Theta(\cdot,s)),
\varphi_{kj} \right\rangle_{X} \psi^\prime(s) ds  \varphi_{kj}(x).
\end{aligned}
\end{equation*}


\section{Example}
\label{sec:06}

Let us take $\Omega = \, ]0,1[$ and $X = L^2(]0,1[)$. 
Consider the following time-fractional evolution diffusion system:
\begin{equation}
\label{exp}
\begin{cases}
\mathcal{^CD}^{\frac{3}{2},\psi(t)}_{0^{+}}\Theta(x,t)=A \Theta(x,t)+f(t ,\Theta(x,t)) \  
&\ \text{ in }  \ {\Omega \times [0,1]}, \\
\Theta(\xi, t)=0 \ &\ \text{ on }  \ \partial \Omega \times [0,1],\\ 		
\Theta(x,0)=\Theta_0(x) \ ;\dfrac{\partial}{\partial t} \Theta(x,0)=\Theta_{1}(x) \ 
& \ \text{ in }  \ \Omega,
\end{cases}
\end{equation}
where 
\begin{equation*}
\begin{split}	
f(t,\Theta(x,t))
&= \dfrac{1}{5}e^{-5t} \Theta(x,t) + t^{4} \sin(\Theta(x,t)),\\ 
\psi(t) &=\sqrt{t+2},
\end{split}
\end{equation*}
and $A=\Delta =\dfrac{\partial^2}{\partial x^2}$ 
is the Laplace operator. The Laplace operator 
$ \Delta $, together with Dirichlet boundary conditions, 
is considered with the domain
$$ 
\mathcal{D}(A)=\lbrace \sigma \in  \operatorname{H}_0^1(\Omega) \ , \ 
A \sigma \in L^2(\Omega) \rbrace 
= \operatorname{H}^2(\Omega)\cap \operatorname{H}_0^1(\Omega).
$$
In this case the operator $\Delta$  generates a uniformly bounded 
strongly continuous cosine family ${\lbrace C(t)\rbrace}_{t\geq0}$.
It is well known that $A$ has a discrete spectrum and its eigenvalues are  
$ \lambda_j $ for every $j\in \mathbb{N}\setminus \{0\}$ 
with the corresponding normalized eigenvectors $ \xi_j(x)$ 
for every $j\in \mathbb{N}\setminus \{0\}$.

We know that $\Vert C(t) \Vert_{L(X)} \leq 1 $    
and $S(t)$ is compact for all $t\geq  0$ \cite{zhou.new}. 
Hence, (H4) is true and we can take $\mathcal{K} =1$.

Let us consider $ \Theta ,\tilde{\Theta}(\cdot,t) \in C([0,1];X) $ 
and $t\in[0,1]$. We can see that
$$
\begin{aligned}
\Vert f&(t,\Theta(\cdot,t))-f(t,\tilde{\Theta}(\cdot,t))\Vert_X\\ 
&\leq \frac{1}{5}e^{-5t} \Vert \Theta(\cdot,t)-\tilde{\Theta}(\cdot,t) \Vert_X
+ t^{4}\Vert \sin(\Theta(\cdot,t))-\sin(\tilde{\Theta}(\cdot,t))\Vert_X\\
&\leq \frac{6}{5} \Vert \Theta(\cdot,t)-\tilde{\Theta}(\cdot,t) \Vert_X,
\end{aligned}
$$
which implies that condition (H1) is satisfied with $N_0 = \dfrac{6}{5}$.
Note that (H2) holds. Indeed,
$$
\frac{ N_0 \mathcal{K} (\psi(T)-\psi(0))^{2q} }{ \Gamma(2q+1)} 
=\frac{\frac{6}{5}(\sqrt{3}-\sqrt{2})^{\frac{3}{2}}}{\Gamma(\frac{5}{2})} < 1 .
$$
Since $B_\beta\subset C([0,1];X)$, for all $\beta$, we can also choose 
$N_\beta = N_0 = \dfrac{6}{5}$, which means that (H3) is true. 
Let $\Theta$ be an element of $B_\beta$. Then,
$$
\begin{aligned}
\Vert f(t,\Theta(\cdot,t)) \Vert_X 
&=\left\Vert \frac{1}{5}e^{-5t} \Theta(\cdot,t) 
+ t^{4} \sin(\Theta(\cdot,t))\right\Vert_X\\
&\leq \frac{1}{5}e^{-5t} \Vert \Theta(\cdot,t)\Vert_X  
+ t^{4} \Vert \sin(\Theta(\cdot,t))\Vert_X\\
&\leq \frac{1}{5}e^{-5t} \Vert \Theta(\cdot,t)\Vert_X + t^{4} \pi,
\end{aligned}
$$
which means that
\begin{equation}
\sup_{\Vert \Theta(\cdot,t) \Vert_X 
\leq  \beta }\Vert f(t,\Theta(\cdot,t)) \Vert_X 
\leq \frac{1}{5}e^{-5t} \beta + t^{4} \pi =:h_\beta(t).
\end{equation}	
Hence,
\begin{equation}
\lim_{\beta\longrightarrow + \infty } 
\displaystyle\sup_{t\in ]0,1]} 
\frac{\Vert h_\beta(t) \Vert_{L^{\infty}}}{\beta} 
= \frac{1}{5}=:\mathcal{H}. 
\end{equation}
Therefore the condition (H5) holds. This yields
$$
\frac{ \mathcal{H} \mathcal{K} (\psi(T)-\psi(0))^{2q} }{ \Gamma(2q+1)} 
=\frac{\frac{1}{5}(\sqrt{3}-\sqrt{2})^{\frac{3}{2}}}{\Gamma(\frac{5}{2})} < 1 .
$$
Thus, according to either Theorem~\ref{jjjj} or Theorem~\ref{jjj}, 
the system \eqref{exp} has a unique mild solution, given by the following formula: 
\begin{equation*}
\begin{aligned}
\Theta(x,t)
&=\displaystyle  \sum_{k=1}^{+\infty}  
E_{2q,1}\left( (\sqrt{t+2}-\sqrt{2})^{2q} \lambda_k\right) 
\left\langle \Theta_0 , \varphi_{ k} \right\rangle_{X} \varphi_{ k}(x)\\
&\quad+ \left[2\sqrt{2}(\sqrt{t+2}-\sqrt{2}) \right]  
\sum_{k=1}^{+\infty}  \left\langle \Theta_{1} , \varphi_{k} \right\rangle_{X} 
\varphi_{k}(x) \\
&\qquad\quad \times E_{2q,2}\left( (\sqrt{t+2}-\sqrt{2})^{2q} \lambda_k \right)\\
&\quad +\displaystyle \sum_{k=1}^{+\infty} \int_0^t  (\sqrt{t+2}-\sqrt{s+2})^{2q-1}  
E_{2q,2q}\left(\lambda_{k} (\sqrt{t+2}-\sqrt{s+2})^{2q}\right) \\
&\qquad\qquad\qquad \times \left\langle 
f(s,\Theta(\cdot,s)),\varphi_{k} \right\rangle_{X} 
\dfrac{\varphi_{k}(x)}{2\sqrt{s+2}} ds.
\end{aligned}
\end{equation*}


\section{Conclusion}
\label{sec:07}

In this manuscript, we have formulated a mild solution for a class of fractional 
evolution systems featuring a $\psi$-Caputo fractional derivative with an order 
within the range of $]1,2]$. Our approach involves utilizing the Laplace 
transform with respect to the $\psi$-function. Additionally, we have highlighted 
the connection between the mild solution of the evolution system and the probability 
density function associated with it, enhancing our comprehension of fractional calculus.
Moreover, by employing fixed-point arguments, we have established conditions that ensure 
both the existence and uniqueness of the mild solution within a suitable spatial framework. 
To show the practical applicability of our theoretical conditions, we presented 
an illustrative example.


\section*{Declarations}

\subsection*{Funding}

This work was partially funded by Funda\c{c}\~{a}o para a Ci\^{e}ncia 
e a Tecnologia, I.P. (FCT, Funder ID = 50110000187) under 
CIDMA (\url{https://ror.org/05pm2mw36}) grants UID/04106/2025. 
Tajani and Torres are also supported through project CoSysM3, 
Reference 2022.03091.PTDC (\url{https://doi.org/10.54499/2022.03091.PTDC}), 
financially supported by national funds (OE) through FCT/MCTES.


\subsection*{Competing interests}

The authors have no competing interests 
to declare that are relevant to the content of this article.

\begin{acknowledgements}
The authors would like to express their gratitude to reviewers
for several constructive comments.
\end{acknowledgements}



\end{document}